\documentclass[11pt]{article}

\usepackage{amsmath}
\usepackage{amsfonts}
\usepackage{amssymb}
\usepackage{amsthm}

\begin{document}

\newtheorem{Theorem}{Theorem}[section]
\newtheorem{Conjecture}[Theorem]{Conjecture}
\newtheorem{Proposition}[Theorem]{Proposition}
\newtheorem{Lemma}[Theorem]{Lemma}
\newtheorem{Corollary}[Theorem]{Corollary}
\newtheorem{Remark}[Theorem]{Remark}
\newtheorem{Example}[Theorem]{Example}

\newcommand{\ord}{\mathop{\rm ord}\nolimits}
\newcommand{\ddiv}{\mathop{\rm div}\nolimits}

\newcommand{\mm}{\mathop{\rm mod}\nolimits}
\newcommand{\Spec}{\mathop{\rm Spec}\nolimits}
\newcommand{\gr}{\mathop{\rm gr}\nolimits}
\newcommand{\tr}{\mathop{\rm tr}\nolimits} 
\newcommand{\Cor}{\mathop{\rm Cor}\nolimits}
\newcommand{\Coker}{\mathop{\rm Coker}\nolimits}
\newcommand{\Image}{\mathop{\rm Image}\nolimits}
\newcommand{\tors}{\mathop{\rm tors}\nolimits}
\newcommand{\Res}{\mathop{\rm Res}\nolimits}
\newcommand{\Gal}{\mathop{\rm Gal}\nolimits}
\newcommand{\Div}{\mathop{\rm Div}\nolimits}
\newcommand{\Tor}{\mathop{\rm Tor}\nolimits}
\newcommand{\Fitt}{\mathop{\rm Fitt}\nolimits}
\newcommand{\Ann}{\mathop{\rm Ann}\nolimits}
\newcommand{\Cl}{\mathop{\rm Cl}\nolimits}
\newcommand{\ii}{\mathop{\rm Image}\nolimits}
\newcommand{\re}{\mathop{\rm Re}\nolimits}
\newcommand{\Sel}{\mathop{\rm Sel}\nolimits}
\newcommand{\cc}{\mathop{\rm char}\nolimits}
\newcommand{\Gel}{\mathop{\mathfrak l}\nolimits}
\newcommand{\Gp}{\mathop{\mathfrak p}\nolimits}
\newcommand{\GO}{\mathop{\mathfrak O}\nolimits}
\newcommand{\GA}{\mathop{\mathfrak A}\nolimits}
\newcommand{\GB}{\mathop{\mathfrak B}\nolimits}
\newcommand{\GH}{\mathop{\mathfrak H}\nolimits}
\newcommand{\GW}{\mathop{\mathfrak W}\nolimits}
\newcommand{\GX}{\mathop{\mathfrak X}\nolimits}
\newcommand{\ZZ}{\mathop{\mathbb Z}\nolimits}
\newcommand{\QQ}{\mathop{\mathbb Q}\nolimits}
\newcommand{\OO}{\mathop{\mathcal O}\nolimits}

%\newcommand{\qed}
%{{\unskip\nobreak\hfil\penalty50\quad\null\nobreak\hfil
%{Q.E.D.}\parfillskip0pt\finalhyphendemerits0\par\medskip}}

\newcommand{\rank}{\mathop{\rm rank}\nolimits}
\newcommand{\image}{\mathop{\rm Im}\nolimits}
\newcommand{\sgn}{\mathop{\rm sgn}\nolimits}
\newcommand{\Hom}{\mathop{\mbox{\rm Hom}}\nolimits}
\newcommand{\Ker}{\mathop{\mbox{\rm Ker}}\nolimits}
\newcommand{\Br}{\mathop{\mbox{\rm Br}}\nolimits}
\newcommand{\mapdown}[1]{\Big\downarrow
 \rlap{$\vcenter{\hbox{$\scriptstyle#1$}}$ }}

\newcommand{\maru}{\mathrel{
\setlength{\unitlength}{.1cm}
\begin{picture}(1.5,2.5)(0,0)
\put(0.75,1){\circle{0.75}}
\end{picture}
}}

\title{
Notes on the dual of the ideal class groups of CM-fields}
%Ideal class groups of CM-fields and the duals as Galois modules}
%Some stronger versions of Brumer's conjecture} 
\author{Masato Kurihara}
\date{}
\maketitle

\begin{abstract}
In this paper, for a CM abelian extension $K/k$ of number fields, 
we propose a conjecture which describes completely the Fitting ideal of 
the minus part of the Pontryagin dual of the $T$-ray class group of $K$ 
for a set $T$ of primes as a $\Gal(K/k)$-module. 
%An essential point is 
Here, we emphasize 
that we consider the full class group, and 
do not throw away the ramifying primes 
(namely, the object we study is {\it not} 
the quotient of the class group by the subgroup generated by 
the classes of ramifying primes). 
We prove that our conjecture is a consequence of 
the equivariant Tamagawa number conjecture, 
and also prove that the Iwasawa theoretic version of 
our conjecture holds true under the assumption $\mu=0$
without assuming eTNC.
\end{abstract}

\section{Introduction} \label{Intro}

It is an important and central theme in number theory to 
pursue the relationship between the arithmetic objects such as  
class groups of number fields and the analytic objects such as 
values of $L$-functions. 
Let $k$ be a totally real number field and 
$K/k$ a finite abelian extension 
with Galois group $G=\Gal(K/k)$ 
such that $K$ is a CM-field. 
Then the Stickelberger element $\theta_{K,S}(0)$
%$\in \QQ[G]$ 
(for the definition, see (\ref{EquivariantSTheta})) is 
related to the class group of $K$, which we regard as a $G$-module.
For example, Brumer's conjecture says that, roughly speaking, 
the Stickelberger element is 
in the annihilator of the class group. 
It is also in the Fitting ideal of the class group in several cases, 
and 
the determination of the Fitting ideal of the class group is an 
important subject in Iwasawa theory (\cite{Ku1}). 
%For example, 
If $k=\QQ$, it was proven that 
the Fitting ideal of the class group 
of $K$ is equal to the Stickelberger ideal 
(except the $2$-component, see \cite{KM}). 
However, for a general totally real field $k$, 
the Pontryagin dual of the class group is the right object to 
study the Fitting ideal (see \cite{GreiKuri}). 
In \cite{Grei4} Greither determined the Fitting ideal of 
the dual of the class group assuming 
the equivariant Tamagawa number conjecture 
and that the group of roots of unity is cohomologically trivial. 

For any finite set $S$ of primes of $k$ 
we denote by $S_{K}$ the set of primes of $K$ above $S$.
For a finite set $T$ of primes of $k$ that are unramified in $K$,  
let $Cl_{K}^{T}$ be the $(\Pi_{w \in T_{K}} w)$-ray class group of $K$. 
In this paper we 
study $Cl_{K}^{T}$ and 
generalize the main result in \cite{Grei4} to $Cl_{K}^{T}$ 
(see Corollary \ref{MainCorollary}).
We note that we do not assume the cohomological triviality of the 
group of roots of unity as in \cite{Grei4}. 

Let $S$ be a finite set of primes of $k$ containing all 
infinite primes and ramifying primes in $K$ such that 
$S \cap T= \emptyset$. 
We denote by $Cl_{K,S}^{T}$ the quotient of $Cl_{K}^{T}$
by the subgroup generated by 
the classes of finite primes in $S_{K}$. 
Burns, Sano and the author proved as a special case
of Theorem 1.5 (i) in \cite{BKS1} that 
the equivariant Tamagawa number conjecture (``eTNC" in short) implies 
that the Fitting ideal of a certain Selmer module is generated by 
the Stickelberger element $\theta_{K,S}^{T}$, and  
$Cl_{K,S}^{T}$ appears as a subgroup of the Selmer module. 
Since $Cl_{K,S}^{T}$ is a subgroup, this does not give 
information on the Fitting ideal of $Cl_{K,S}^{T}$ in general. 
Also, $Cl_{K,S}^{T}$ is smaller than the full class group 
$Cl_{K}^{T}$ which we want to study.

In order to overcome these difficulties we use the beautiful ideas  
in Greither's paper \cite{Grei4}.
An important idea in \cite{Grei4} which we also use here is 
to use ``the local modules" $W_{K_{w}}$, $W_{v}$ 
by Gruenberg and Weiss \cite{GruW}, which we will introduce 
in \S \ref{sec2}.
In this sense, this paper heavily relies on the ideas in \cite{Grei4}.  
A new idea in this paper is to consider a Tate sequence using 
linear duals 
$M^{\circ}=\Hom(M, \ZZ)$ of modules $M$ 
(see the exact sequences in Proposition \ref{PropCalS} 
and Proposition \ref{P2}). 

In \S \ref{sec2} we introduce homomorphisms $\psi_{S}$, $\psi$ 
for a general Galois extension of number fields.
The Pontryagin dual of $Cl_{K}^{T}$ appears as the cokernel of 
the linear dual $\psi^{\circ}$ of $\psi$ (see Proposition \ref{P2}). 
The complex $\GA \stackrel{\psi_{S}}{\longrightarrow} \GB$ 
represents $R\Gamma_{T}(\OO_{K,S}, {\mathbb G}_{m})$ in \cite{BKS1}.
We compare in \S \ref{ComparisonofTwoHom} 
the two homomorphisms $\psi_{S}^{\circ}$, $\psi^{\circ}$ 
in order to get information on $Cl_{K}^{T}$.

In \S \ref{ComparisonofTwoHom} we propose Conjecture \ref{C1} which 
describes completely the Fitting ideal of 
the minus part of the Pontryagin dual of 
$Cl_{K}^{T}$, and prove it assuming Conjecture \ref{C2} 
which is a conjecture on the homomorphisms $\psi_{S}$. 
We show that eTNC implies 
Conjecture \ref{C2} (see Proposition \ref{eTNCimpliesC2}), 
so also implies Conjecture \ref{C1} 
(see Corollary \ref{MainCorollary}).  
We use eTNC in the style of \cite{BKS1} in Proposition \ref{eTNCimpliesC2}.
%by Burns, Sano and the author, and examine the relation between the maps
%$\psi_{S}$ and some complexes in \cite{BKS1}.  

In \S \ref{sec4} we also prove, without assuming eTNC, 
Theorem \ref{MT2} which is 
the Iwasawa theoretic version of 
Conjecture \ref{C1}, and which determines completely  
the Fitting ideal of the $T$-modified Iwasawa modules. 
Theorem \ref{MT2} can be regarded as 
a refinement of a result by Greither and 
Popescu \cite{GreiPo}, and a generalization of a result in \cite{Ku3} 
(see Remark \ref{FinalRemark}).

The author would like to thank Cornelius Greither for discussion 
with him on the subjects in this paper, and for his giving the author 
some comments after reading the first draft of this paper.

\section{$T$-class groups of number fields as Galois modules} \label{sec2}

\subsection{A homomorphism $\psi_{S}: \GA \longrightarrow \GB$} 
\label{ConstructionofpsiS} 

In this subsection we suppose that 
$K/k$ is a finite Galois extension of number fields with 
$G=\Gal(K/k)$.

The goal of this subsection is to define 
two $\ZZ[G]$-modules 
$\GA$, $\GB$, and a homomorphism $\psi_{S}:\GA \longrightarrow \GB$, 
which represents $R\Gamma_{T}(\OO_{K,S}, {\mathbb G}_{m})$ in \cite{BKS1}. 
%and propose a conjecture on this homomorphism. 

For any finite set $S$ of primes of $k$ we denote by $S_{K}$ the 
set of primes of $K$ above $S$. 
Let $S_{\infty}$ be the set of all infinite primes of $k$. 
For any finite set $S$ of primes of $k$ 
such that $S \supset S_{\infty}$ we denote by 
$\OO_{K,S}$ the subring of $K$ consisting of integral elements 
outside $S$. 
The integer ring $\OO_{K,S_{\infty}}$ is denoted by $\OO_{K}$. 

We take and fix a finite set $T$ of finite primes of $k$ that are 
unramified in $K$ such that 
$(\OO_{K}^{T})^{\times}=\{x \in \OO_{K}^{\times} \mid 
x \equiv 1 \ \mbox{(mod $w$) for all primes $w$ above $T$}\}$ 
is $\ZZ$-torsion free. 

For a finite set $S$ of primes of $k$ such that
$S \supset S_{\infty}$ and $S \cap T = \emptyset$, 
we define 
$$(\OO_{K,S}^{T})^{\times}=\{x \in \OO_{K,S}^{\times} \mid 
x \equiv 1 \ \mbox{(mod $w$) for all primes $w \in T_{K}$}\}$$ 
and $Cl_{K,S}^{T}$ to be  
the ray class group of $\OO_{K,S}$ modulo $\Pi_{w \in T_{K}} w$. 

We define a subgroup $J_{K,S}^{T}$ of the id\`{e}le group of $K$ by
$$J_{K,S}^{T}=\prod_{w \in T_{K}} U_{K_{w}}^{1} \times 
\prod_{w \not \in (S \cup T)_{K}} U_{K_{w}} \times 
\prod_{w \in S_{K}}K_{w}^{\times}.$$

Let $S_{{\rm ram}}(K/k)$ be the set of 
all ramifying finite primes in $K/k$.
From now on we fix a finite set $S$ of primes of $k$ such that
$S \supset S_{\infty} \cup S_{{\rm ram}}(K/k)$ and 
$S \cap T = \emptyset$. 

We also take  
a finite set $S'$ of primes of $k$ 
such that (i) $S' \supset S$, 
%(ii) $S'$ is disjoint from $T$, 
(ii) $Cl_{K,S'}^{T}=0$ 
%where $Cl_{K,S'}^{T}$ is the ray class group of $\OO_{K,S'}$ modulo $\Pi_{w \in T_{K}} w$, 
and that 
(iii) the decomposition groups $G_{v}$ of $v$ for all $v \in S'$ 
generate $G$.

Let $C_{K}$ be the idele class group of $K$. 
By definitions, we have an exact sequence 
\begin{equation} \label{e1}
0 \longrightarrow (\OO_{K,S}^{T})^{\times} \longrightarrow 
J_{K,S}^{T} \longrightarrow C_{K} \longrightarrow 
Cl_{K,S}^{T} \longrightarrow 0.
\end{equation}
From our assumption (ii) above, 
we also have an exact sequence 
$$0 \longrightarrow (\OO_{K,S'}^{T})^{\times} \longrightarrow 
J_{K,S'}^{T} \longrightarrow C_{K} \longrightarrow 0$$
for $S'$.

For any group ${\mathcal G}$, we denote by $\Delta {\mathcal G}$ 
the augmentation ideal in $\ZZ[{\mathcal G}]$. 
For a prime $w$ of $K$, 
we denote by $G_{w}$, $I_{w}$ the decomposition subgroup and 
the inertia subgroup of $w$ in $G$.
We consider $V_{K_{w}}$ the extension of 
$\Delta G_{w}$ by $K_{w}^{\times}$ corresponding to the local 
fundamental class (see Gruenberg and Weiss \cite{GruW}, 
Greither \cite{Grei4});
$0 \longrightarrow K_{w}^{\times} \longrightarrow V_{K_{w}} 
\longrightarrow \Delta G_{w} \longrightarrow 0$.
If $w$ is a finite prime, we define $W_{K_{w}}$ by 
$W_{K_{w}}=\Coker(U_{K_{w}} \longrightarrow V_{K_{w}})$. 
Thus we have an exact sequence 
$$ 0 \longrightarrow \ZZ \longrightarrow W_{K_{w}} \longrightarrow 
\Delta G_{w} \longrightarrow 0.$$
%which corresponds to the extension of 
%$\Delta G_{w}$ by $\ZZ$. 
More explicitly, as in (23) on page 1412 in Greither \cite{Grei4},
one can write 
$$W_{K_{w}}= \Ker (\Delta G_{w} \times \ZZ[G_{w}/I_{w}] 
\longrightarrow
\ZZ[G_{w}/I_{w}])$$
where the above homomorphism is defined by $(x,y) \mapsto 
\overline{x} + ({\mathcal F}_{w}^{-1}-1)y$ with 
$\overline{x}=x$ mod $I_{w}$ and the Frobenius ${\mathcal F}_{w}$ of $w$ 
in $G_{w}/I_{w}$. 
Note that we are using a homomorphism which is slightly 
different from (23) in \cite{Grei4}.
This modification is necessary 
to get good bases of $\GB$ and $W_{S_{\infty}}^{\circ} 
\otimes \QQ$ later.
We note that if $w$ is unramified in $K/k$, $I_{w}=0$ and 
the projection to the second 
component $(x,y) \mapsto y$ gives an isomorphism 
$W_{K_{w}} \simeq \ZZ[G_{w}]$.

We put 
$$V_{S'}^{T}=\prod_{w \in T_{K}} U_{K_{w}}^{1} \times 
\prod_{w \not \in (S' \cup T)_{K}, } U_{K_{w}} \times 
\prod_{w \in (S')_{K}} V_{K_{w}},$$ 
and $W_{S'}=V_{S'}^{T}/J_{K,S'}^{T}$, $W_{S}=V_{S'}^{T}/J_{K,S}^{T}$. 
So we have $W_{S'}=\prod_{w \in S'_{K}} \Delta G_{w}$, and
$$W_{S}=
\prod_{w \in S_{K}} \Delta G_{w} \times 
\prod_{w \not \in (S' \setminus S)_{K}} W_{K_{w}}
=\prod_{w \in S_{K}} \Delta G_{w} \times 
\prod_{w \not \in (S' \setminus S)_{K}} \ZZ[G_{w}].$$
where we used the isomorphisms
$W_{K_{w}} \simeq \ZZ[G_{w}]$ for $w \in (S' \setminus S)_{K}$ 
which we defined in the previous paragraph
to get the second equality 
(note that primes in $S' \setminus S$ are unramified).
 
Let 
$$0 \longrightarrow C_{K} \longrightarrow \GO \longrightarrow 
\Delta G \longrightarrow 0$$
be the extension corresponding to the global fundamental class as 
in \cite{Grei4}, 
and consider the commutative diagram of exact sequences;
$$
\begin{array}{ccccccccc}
0 &         \longrightarrow  & J_{K,S'}^{T} & \longrightarrow &
V_{S'}^{T} & \longrightarrow  & W_{S'} & \longrightarrow & 0\\
&&\mapdown{}&&\mapdown{}&&\mapdown{}&& \\
0 &         \longrightarrow  & C_{K}  & \longrightarrow & 
\GO&       \longrightarrow & \Delta G & \longrightarrow &0. 
\end{array}
$$
The conditions (ii) and (iii) imply that the left and right  
vertical maps in the diagram are surjective (see also 
the exact sequence (\ref{e1})). 
Therefore, the central vertical map is also surjective 
(see \cite{Grei4} page 1409).
We next consider the commutative diagram 
$$
\begin{array}{ccccccccc}
0 &         \longrightarrow  & J_{K,S}^{T} & \longrightarrow &
V_{S'}^{T} & \longrightarrow  & W_{S} & \longrightarrow & 0\\
&&\mapdown{}&&\mapdown{}&&\mapdown{}&& \\
0 &         \longrightarrow  & C_{K}  & \longrightarrow & 
\GO&       \longrightarrow & \Delta G & \longrightarrow &0, 
\end{array}
$$
which is obtained by replacing $J_{K,S'}^{T}$ by $J_{K,S}^{T}$.
We put $A=\Ker(V_{S'}^{T} \longrightarrow \GO)$ and 
$W'_{S}=\Ker(W_{S} \longrightarrow \Delta G)$. 
By the exact sequence (\ref{e1}) and the snake lemma, 
we have an exact sequence 
\begin{equation} \label{e02}
0 \longrightarrow (\OO_{K,S}^{T})^{\times} \longrightarrow 
A \longrightarrow W'_{S} \longrightarrow 
Cl_{K,S}^{T} \longrightarrow 0.
\end{equation}

We put $\GB=\prod_{w \in (S')_{K}} \ZZ[G_w]$, and regard $W_{S}$ as a submodule 
of $\GB$. 
By definition $\GB/W_{S} \simeq \prod_{w \in S_{K}} \ZZ$.
The map $W_{S}  \longrightarrow   \Delta G$ can be extended 
to $\GB \longrightarrow   \ZZ[G]$.
Since $S$ is non-empty, this is surjective. 
We denote by $B$ the kernel of this homomorphism.
Now we have a commutative diagram of exact sequences:
$$
\begin{array}{ccccccccc}
&& 0        && 0        &&   0      && \\
&&\mapdown{}&&\mapdown{}&&\mapdown{}&& \\
0 &         \longrightarrow  & W'_{S} & \longrightarrow &
W_{S} & \longrightarrow  & \Delta G & \longrightarrow & 0\\
&&\mapdown{}&&\mapdown{}&&\mapdown{}&& \\
0 &         \longrightarrow  & B & \longrightarrow &
\GB & \longrightarrow  & \ZZ[G] & \longrightarrow & 0\\
&&\mapdown{}&&\mapdown{}&&\mapdown{}&& \\
0 &         \longrightarrow  & \GX_{K,S}  & \longrightarrow & 
\prod_{w \in S_{K}} \ZZ &       \longrightarrow & \ZZ & \longrightarrow &0\\
&&\mapdown{}&&\mapdown{}&&\mapdown{}&& \\ 
&& 0        && 0        &&   0      &&
\end{array}
$$
where $\GX_{K,S}$ is the kernel of the homomorphism 
$\prod_{w \in S_{K}} \ZZ \longrightarrow \ZZ$.

The map $A \longrightarrow W_{S}'$ obtained above defines 
a map $\psi_{S}: A \longrightarrow B$ by regarding $W_{S}'$ 
as a submodule of $B$.
We define ${\mathcal H}_{K,S}^{T}$ to be the cokernel of 
$\psi_{S}:  A \longrightarrow B$.
Thus we have obtained the following.

\begin{Proposition} \label{PropCalH}
The homomorphism $\psi_{S}: A \longrightarrow B$ obtained above has 
kernel $(\OO_{K,S}^{T})^{\times}$ and cokernel ${\mathcal H}_{K,S}^{T}$ 
for which, we have an exact sequence 
$$0 \longrightarrow Cl_{K,S}^{T} \longrightarrow 
{\mathcal H}_{K,S}^{T} \longrightarrow 
\GX_{K,S} \longrightarrow 0.$$
The module $B$ is a finitely generated free $\ZZ[G]$-module. 
\end{Proposition}
  
\begin{Remark}
\begin{rm}
The module ${\mathcal H}_{K,S}^{T}$ is isomorphic to the module 
${\cal S}_{S,T}^{{\rm tr}}({\mathbb G}_{m/K})$ 
constructed in \cite{BKS1} Definition 2.6 
by Burns, Sano and the author.
This module is also regarded as the ``Weil \'{e}tale 
cohomology group $H_{T}^2(\OO_{K,S}, \ZZ(1))$". 
We note that the assumption $S \supset S_{{\rm ram}}(K/k)$ 
is important to get this Proposition.
\end{rm}
\end{Remark}

Note that the middle horizontal exact sequence in the diagram before 
Proposition \ref{PropCalH} splits. 
So we have an isomorphism $\GB \simeq B \oplus \ZZ[G]$. 
Therefore, putting $\GA=A \oplus \ZZ[G]$, we can define 
$\GA \longrightarrow \GB$ 
which is an extension of $A \longrightarrow B$, and 
whose kernel and cokernel coincide with 
the kernel and cokernel of $\psi_{S}: A \longrightarrow B$, respectively.
We denote this map also by $\psi_{S}: \GA \longrightarrow \GB$.

For any $\ZZ[G]$-module $M$, we denote the linear dual by 
$M^{\circ}=\Hom(M, \ZZ)$, and the Pontryagin dual 
by $M^{\vee}=\Hom(M,\QQ/\ZZ)$.
Taking the linear dual of $\psi_{S}: \GA \longrightarrow \GB$, 
we have 
$\psi_{S}^{\circ}: \GB^{\circ} \longrightarrow \GA^{\circ}$, 
whose cokernel we denote by 
${\cal S}_{K,S}^{T}$. 
Of course, this is isomorphic to the cokernel of 
$B^{\circ} \longrightarrow A^{\circ}$.

\begin{Proposition} \label{PropCalS}
The kernel of $\psi_{S}^{\circ}: \GB^{\circ} \longrightarrow \GA^{\circ}$ 
is isomorphic to $\GX_{K,S}^{\circ}$, and the cokernel 
${\cal S}_{K,S}^{T}$ sits in an exact sequence
$$0 \longrightarrow (Cl_{K,S}^{T})^{\vee} \longrightarrow 
{\mathcal S}_{K,S}^{T} \longrightarrow 
((\OO_{K,S}^{T})^{\times})^{\circ} \longrightarrow 0.$$
\end{Proposition}

This module ${\mathcal S}_{K,S}^{T}$ is isomorphic to 
the module ${\mathcal S}_{S, T}({\mathbb G}_{m/K})$ 
in \cite{BKS1} Definition 2.1.
One can regard $0 \longrightarrow \GX_{K,S}^{\circ} 
\longrightarrow
\GB^{\circ} \longrightarrow \GA^{\circ} 
\longrightarrow
{\cal S}_{K,S}^{T} \longrightarrow 0$
as a (linear dual version of) Tate sequence.

\begin{proof}
We denote by $M$ the image of $\psi_{S}: \GA  \longrightarrow \GB$.
Then $0 \longrightarrow M^{\circ} \longrightarrow 
\GA^{\circ} \longrightarrow ((\OO_{K,S}^{T})^{\times})^{\circ}
\longrightarrow 0$ and 
$$0 \longrightarrow \GX_{K,S}^{\circ} \longrightarrow 
\GB^{\circ} \longrightarrow M^{\circ}
\longrightarrow 
{\rm Ext}^1({\mathcal H}_{K,S}^{T}, \ZZ)=(Cl_{K,S}^{T})^{\vee}
\longrightarrow 0$$
are both exact since the torsion part of ${\mathcal H}_{K,S}^{T}$ is
$Cl_{K,S}^{T}$ and the quotient ${\mathcal H}_{K,S}^{T}/Cl_{K,S}^{T}$
is isomorphic to $\GX_{K,S}$.
Thus $\psi_{S}^{\circ}$ has kernel isomorphic to $\GX_{K,S}^{\circ}$.
Also, concerning the cokernel ${\cal S}_{K,S}^{T}$, we get 
the exact sequence in Proposition \ref{PropCalS} from 
the above two exact sequences.
\end{proof}

\subsection{A homomorphism $\psi$} \label{CM-extensionsSubsection}

From now on we assume that $K/k$ is a finite abelian extension 
such that 
$k$ is totally real and $K$ is a CM-field as in \S \ref{Intro}.  
We will define a homomorphism $\psi^{\circ}: 
(W_{S_{\infty}}^{\circ})^-_{\ZZ_{p}} 
\longrightarrow 
(\GA^{\circ}_{\ZZ_{p}})^-$ and study it. 
 
We consider $J_{K, S_{\infty}}^{T}$, and get an exact sequence 
\begin{equation} \label{e11}
0 \longrightarrow (\OO_{K}^{T})^{\times} \longrightarrow 
J_{K,S_{\infty}}^{T} \longrightarrow C_{K} \longrightarrow 
Cl_{K}^{T} \longrightarrow 0
\end{equation}
%as we got (\ref{e1}) in the previous subsection.
from definitions.
We define $W_{S_\infty}=V_{S'}^{T}/J_{K,S_{\infty}}^{T}$, so
$$W_{S_\infty}=
%V_{S'}^{T}/J_{K,S_{\infty}}^{T}=
\prod_{w \in (S' \setminus S_{\infty})_{K}} W_{K_{w}} \times 
\prod_{w \in (S_{\infty})_{K}} \Delta G_{w}.$$
From the commutative diagram
$$
\begin{array}{ccccccccc}
0 &         \longrightarrow  & J_{K,S_{\infty}}^{T} & \longrightarrow &
V_{S'}^{T} & \longrightarrow  & W_{S_{\infty}} & \longrightarrow & 0\\
&&\mapdown{}&&\mapdown{}&&\mapdown{}&& \\
0 &         \longrightarrow  & C_{K}  & \longrightarrow & 
\GO&       \longrightarrow & \Delta G & \longrightarrow &0, 
\end{array}
$$
defining 
$W_{S_{\infty}}'=\Ker(W_{S_{\infty}} \longrightarrow \Delta G)$,
we have an exact sequence 
\begin{equation} \label{e22}
0 \longrightarrow (\OO_{K}^{T})^{\times} \longrightarrow 
A \longrightarrow W'_{S_{\infty}} \longrightarrow 
Cl_{K}^{T} \longrightarrow 0
\end{equation}
as we got the exact sequence (\ref{e02}) in the previous subsection.

Now we take an odd prime number $p$, and study the $p$-components 
of the above modules.  
For any $\ZZ[G]$-module $M$ we write $M_{\ZZ_{p}}=M \otimes \ZZ_{p}$
and denote by $M^{-}_{\ZZ_{p}}$ the minus part of $M_{\ZZ_{p}}$ 
(which consists of elements on which the complex conjugation acts as $-1$). 

Since $(\Delta G)^-_{\ZZ_{p}}=\ZZ_{p}[G]^-$, the sequence 
$0 \longrightarrow (W'_{S_{\infty}})^-_{\ZZ_{p}} 
\longrightarrow (W_{S_{\infty}})^-_{\ZZ_{p}} \longrightarrow 
(\Delta G)^-_{\ZZ_{p}} \longrightarrow 0$ splits 
as an exact sequence of $\ZZ_{p}[G]^-$-modules. 
Therefore, putting $\GA=A \oplus \ZZ[G]$ as in the previous 
subsection, 
we can construct a map 
$$\psi: \GA^-_{\ZZ_{p}} \longrightarrow (W_{S_{\infty}})^-_{\ZZ_{p}}$$ 
which is an extension of $A^-_{\ZZ_{p}} \longrightarrow 
(W_{S_{\infty}}')_{\ZZ_{p}}^-$, 
whose kernel is $((\OO_{K}^{T})^{\times}_{\ZZ_{p}})^-$, and 
whose cokernel is $((Cl_{K}^{T})_{\ZZ_{p}})^-$.
Since $(\OO_{K}^{T})^{\times}$ is torsion free, 
we have $((\OO_{K}^{T})^{\times}_{\ZZ_{p}})^-=0$. 
Therefore, we have an exact sequence 
\begin{equation} \label{e23}
0 \longrightarrow \GA^-_{\ZZ_{p}} 
\stackrel{\psi}{\longrightarrow} 
(W_{S_{\infty}})^-_{\ZZ_{p}} \longrightarrow 
((Cl_{K}^{T})_{\ZZ_{p}})^- \longrightarrow 0.
\end{equation}

Taking the linear dual of the exact sequence (\ref{e23}), 
we obtain 
$$
0 \longrightarrow (W_{S_{\infty}}^{\circ})^-_{\ZZ_{p}} 
\stackrel{\psi^{\circ}}{\longrightarrow} (\GA^{\circ})^-_{\ZZ_{p}}
\longrightarrow 
((Cl_{K}^{T})^{\vee}_{\ZZ_{p}})^- \longrightarrow 0
$$
because ${\rm Ext}^1_{\ZZ_{p}}
((Cl_{K}^{T})_{\ZZ_{p}})^-, \ZZ_{p})=
((Cl_{K}^{T})^{\vee}_{\ZZ_{p}})^-$.

\vspace{2mm}

For an infinite prime $v \in S_{\infty}$ we consider 
$\Delta_{v}=\bigoplus_{w \mid v} \Delta G_{w}$. 
Here and from now on, 
we use the notation $\bigoplus$ instead of $\prod$.
Since the complex conjugation $\rho$ 
acts as $-1$ on $\Delta G_{w}$, 
$(\Delta_{v})^-_{\ZZ_{p}}=
(\bigoplus_{w \mid v} \Delta G_{w})^-_{\ZZ_{p}}$ is 
a free $\ZZ_{p}[G]^-$-module of rank $1$.
Choosing a prime $w$ above $v$ and taking $e_{v} \in 
(\Delta_{v})^-_{\ZZ_{p}}$ 
whose $w$-component is $\frac{1-\rho}{2}$ and other components 
are zero where $\rho$ is the complex conjugation, 
we have an equality 
$(\Delta_{v})^-_{\ZZ_{p}} = \ZZ_{p}[G]^-e_{v}$.

For a finite prime $v$ in $S'$, we put 
$W_{v}=\bigoplus_{w \mid v} W_{K_{w}}$.
For $w \vert v$, by 
the description of $W_{K_{w}}$ mentioned in the previous subsection, 
we can show that 
$W_{K_{w}}^{\circ}$ 
is isomorphic to the quotient of 
$\ZZ[G_{w}]/(N_{G_{w}}) \oplus \ZZ[G_{w}/I_{w}]$ by 
the submodule generated by 
$(N_{I_{w}} x, ({\mathcal F}_{w}-1)(x))$ 
for all $x \in \ZZ[G/I_{w}]$ (see (24) on page 1412 in \cite{Grei4}).
In this way we regard $W_{K_{w}}^{\circ}$ as a quotient of 
$\ZZ[G_{w}] \oplus \ZZ[G_{w}]$.
The natural map 
$\ZZ[G_{w}] \oplus \ZZ[G_{w}] \longrightarrow 
W_{K_{w}}^{\circ}$ induces  
$c_{w}: \QQ[G_{w}] \oplus \QQ[G_{w}] \longrightarrow 
W_{K_{w}}^{\circ} \otimes \QQ$. 
By Greither \cite{Grei4} Lemma 6.1 
$c_{w}((1,1))$ is a basis of 
$W_{K_{w}}^{\circ} \otimes \QQ$; 
\begin{equation} \label{BasisIsom1}
\QQ[G_{w}]c_{w}((1,1)) =
W_{K_{w}}^{\circ} \otimes \QQ.
\end{equation}
Since we slightly modified the homomorphism used in the 
definition of $W_{K_{w}}$ as we mentioned in the previous subsection, 
we give a proof of (\ref{BasisIsom1}). 
Since $G$ is abelian, ${\mathcal F}_{w}$ and $I_{w}$ are 
independent of the choice of $w$ above $v$, so 
we write $I_{v}$ and ${\mathcal F}_{v}$ for them.
Put 
\begin{equation} \label{gv}
g_{v}=1-{\mathcal F}_{v}+\# I_{v}. 
\end{equation}
This is a nonzero divisor in $\QQ[G_{w}]$.
Since 
$$(0, -g_{v})=(N_{I_{v}}, {\mathcal F}_{v}-1)-(N_{I_{v}}, \# I_{v}),$$
$c_{w}((N_{I_{v}}, {\mathcal F}_{v}-1))=0$ and 
$c_{w}((N_{I_{v}}, N_{I_{v}}))=c_{w}((N_{I_{v}}, \# I_{v}))$, 
we have 
\begin{equation} \label{BasisIsom2}
c_{w}((0,1))=g_{v}^{-1}N_{I_{v}}c_{w}((1,1))
\end{equation}
in $W_{K_{w}}^{\circ} \otimes \QQ$. 
This shows that both $c_{w}((0,1))$ and $c_{w}((1,0))$ 
are in the space generated by $c_{w}((1,1))$ and 
we get $\QQ[G_{w}]c_{w}((1,1)) =
W_{K_{w}}^{\circ} \otimes \QQ$.

\vspace{5mm}

Thus, by fixing $w$ above $v$ and using $c_{w}$, 
we have an isomorphism $\QQ[G] \simeq \bigoplus_{w \mid v} \QQ[G_{w}]$
and a homomorphism
$$c_{v}: \QQ[G] \oplus \QQ[G] \longrightarrow 
\bigoplus_{w \mid v} W_{K_{w}}^{\circ} \otimes \QQ 
=W_{v}^{\circ} \otimes \QQ.$$
We define 
\begin{equation} \label{BasisIsom3}
e_{v}=c_{v}((1,1)) \in W_{v}^{\circ} \otimes \QQ,
\end{equation}
which is a basis of $W_{v}^{\circ} \otimes \QQ$.

In this way we get a basis $(e_{v})_{v \in S'}$ 
of a free $\QQ[G]$-module 
$W_{S_{\infty}}^{\circ} \otimes \QQ$ of rank $\#S'$. 

For a finite prime $v \in S'$ 
we consider the equality (\ref{BasisIsom1}). 
Since $W_{K_{w}}^{\circ}$ is generated by 
$c_{w}((1,1))$ and $c_{w}((0,1))$, 
using (\ref{BasisIsom2}), we have 
$$W_{K_{w}}^{\circ}=(1, \frac{1}{g_{v}}N_{I_{v}})\ZZ[G_{w}]
c_{w}((1,1)).$$
%(see \cite{Grei4} Lemma 8.3). 
Therefore, we get 
$$W_{v}^{\circ}=(1, \frac{1}{g_{v}}N_{I_{v}})\ZZ[G]e_{v}$$ 
where
$g_{v}=1-{\mathcal F}_{v}+\# I_{v}$ as in $(\ref{gv})$. 

Put 
\begin{equation} \label{hv}
h_{v}=(1-\frac{N_{I_{v}}}{\#I_{v}})
+\frac{N_{I_{v}}}{\#I_{v}}g_{v}
\end{equation}
which is a nonzero divisor of $\QQ[G]$ 
as in Lemma 8.3 in Greither \cite{Grei4}. 
Then by this lemma we have
$$
(1, \frac{1}{g_{v}}N_{I_{v}})\ZZ[G]=h_{v}^{-1}
(N_{I_{v}}, 1-\frac{N_{I_{v}}}{\#I_{v}} {\mathcal F}_{v})
\ZZ[G]
$$
because 
$h_{v} = 1-\frac{N_{I_{v}}}{\#I_{v}} {\mathcal F}_{v} + N_{I_{v}}$.
Therefore, we have  
\begin{equation} \label{BasisIsom4}
W_{v}^{\circ} = (1, \frac{1}{g_{v}}N_{I_{v}})\ZZ[G]e_{v}
=h_{v}^{-1}
(N_{I_{v}}, 1-\frac{N_{I_{v}}}{\#I_{v}} {\mathcal F}_{v})e_{v}
\end{equation}
and an isomorphism
$$
W_{v}^{\circ} \simeq 
(N_{I_{v}}, 1-\frac{N_{I_{v}}}{\#I_{v}} {\mathcal F}_{v})
%h_{v}^{-1}
\ZZ[G].
$$

We note that if $v$ is unramified, 
$(N_{I_{v}}, 1-\frac{N_{I_{v}}}{\#I_{v}} {\mathcal F}_{v}) \ZZ[G]
=(1, 1-{\mathcal F}_{v}) \ZZ[G]=\ZZ[G]$. 
Therefore, recalling that $S_{{\rm ram}}(K/k)$ is the set of 
finite primes that are ramified in $K$, we have an isomorphism
$$\bigoplus_{v \in S' \setminus S_{\infty}} 
W_{v}^{\circ} \simeq 
\bigoplus_{v \in S_{{\rm ram}}(K/k)} 
(N_{I_{v}}, 1-\frac{N_{I_{v}}}{\#I_{v}} {\mathcal F}_{v}) \ZZ[G] 
\oplus
\bigoplus_{v \in S' \setminus (S_{\infty} \cup S_{{\rm ram}}(K/k))} 
\ZZ[G].$$

Thus we get information of the Galois module structure of 
$W^{\circ}_{S_{\infty}}$.
We have obtained  
\begin{Proposition} \label{P2}
\begin{equation} \label{e4}
0 \longrightarrow  (W_{S_{\infty}}^{\circ})^-_{\ZZ_{p}} 
\stackrel{\psi^{\circ}}{\longrightarrow} 
(\GA^{\circ}_{\ZZ_{p}})^- 
\longrightarrow ((Cl_{K}^{T})^{\vee}_{\ZZ_{p}})^- \longrightarrow 0
\end{equation}
is exact. 
Here, $(\GA^{\circ}_{\ZZ_{p}})^-$ 
is a free $\ZZ_{p}[G]^-$-module of 
rank $\#S'$, and 
\begin{equation} \label{I3}
(W_{S_{\infty}}^{\circ})^-_{\ZZ_{p}}  
\simeq 
\bigoplus_{v \in S' \setminus S_{{\rm ram}}(K/k)} 
\ZZ_{p}[G]^- \oplus 
\bigoplus_{v \in S_{{\rm ram}}(K/k)} 
(N_{I_{v}}, 1-\frac{N_{I_{v}}}{\#I_{v}} {\mathcal F}_{v})
\ZZ_{p}[G]^-.
\end{equation}
\end{Proposition}

\begin{proof}
The exactness of the sequence (\ref{e4}) 
and the isomorphism (\ref{I3})
were already proved before this proposition. 
Since $A$ is torsion free and cohomologically trivial, 
$\GA^-_{\ZZ_{p}}$ is also cohomologically trivial. 
Note that $\GA^-_{\ZZ_{p}} \otimes_{\ZZ_{p}} \QQ_{p}$ is isomorphic to 
$(W^{\circ})^-_{\ZZ_{p}} \otimes_{\ZZ_{p}} \QQ_{p}$ which is 
free of rank $\#S'$ 
over $\QQ_{p}[G]^-$. 
So $\GA^-_{\ZZ_{p}}$ is 
a free $\ZZ_{p}[G]^-$-module of rank $\#S'$.
%We proved the isomorphism (\ref{I3}) before we state this proposition.  
\end{proof}

\section{Fitting ideals}

\subsection{Stickelberger ideals and a conjecture
on Fitting ideals} \label{MainConjecture}

Let $K/k$ be a finite abelian CM-extension, 
and $G$, $T$,...be as in 
\S \ref{CM-extensionsSubsection}.    
We will first define a certain Stickelberger ideal 
$\Theta^{T}(K) \subset \ZZ[G]$.
 
For a character $\chi$ of $G$, 
we write $L(s, \chi)$ for the primitive $L$-function for $\chi$; 
this function omits exactly the Euler factors of primes dividing the 
conductor of $\chi$. 
We define 
$$\omega^{T}=\sum_{\chi \in {\hat G}} L_{T}(0,\chi^{-1}) \epsilon_{\chi} 
\in \QQ[G]$$
where 
$$L_{T}(s,\chi)=(\prod_{v \in T}(1-\chi({\mathcal F}_{v})N(v)^{1-s}))
L(s, \chi)$$ 
is the $T$-modified $L$-function 
and $\epsilon_{\chi}
=(\# G)^{-1}\sum_{\sigma \in G} \chi(\sigma) \sigma^{-1}$ 
is the idempotent of the $\chi$-component. 
We know that $\omega^{T} \in \QQ[G]$.

As in the previous section 
we denote by $S_{{\rm ram}}(K/k)$ the set of 
all ramifying finite primes in $K/k$.  
For $v \in S_{{\rm ram}}(K/k)$ 
we define a $\ZZ[G]$-module $U_{v}$ in $\QQ[G]$ by 
$$U_{v}=(N_{I_{v}}, 1-\frac{N_{I_{v}}}{\#I_{v}} {\mathcal F}_{v}^{-1})\ZZ[G] \subset 
\QQ[G].$$
We define the Stickelberger ideal $\Theta^{T}(K)$ by 
$$\Theta^{T}(K)=(\prod_{v \in S_{{\rm ram}}(K/k)} U_{v}) \omega^{T}.$$

\begin{Proposition} \label{P3}
This Stickelberger ideal $\Theta^{T}(K)$ is in $\ZZ[G]$, namely 
it is an ideal of $\ZZ[G]$.
\end{Proposition}

\begin{proof} 
For an intermediate field $F$ of $K/k$ and a finite set $S$ 
of finite primes that contains all ramifying primes in $F$, 
we define the equivariant zeta function $\theta_{F,S}(s)$ by 
\begin{equation} \label{EquivariantSTheta}
\theta_{F,S}(s)=\prod_{\chi \in {\hat \Gal}(F/k)} L_{S}(s, \chi^{-1}) 
\epsilon_{\chi}
\end{equation} 
where $L_{S}(s, \chi)$ is the $L$-function 
obtained by removing the Euler factors for all 
$v \in S$.  
We consider its $T$-modification 
$$
\theta^{T}_{F,S}(s)=
(\prod_{v \in T}(1-{\mathcal F}_{v}^{-1}N(v)^{1-s}))\theta_{F,S}(s)
$$
and the $(S,T)$-Stickelberger element 
%$\theta_{F,S}(0) \in \QQ[\Gal(F/k)]$ and its $T$-modification 
\begin{equation} \label{ThetaST}
\theta^{T}_{F,S}=\theta^{T}_{F,S}(0)=
(\prod_{v \in T}(1-{\mathcal F}_{v}^{-1}N(v)))\theta_{F,S}(0).
\end{equation}
It is known by Deligne and Ribet and Cassou-Nogu\`{e}s 
that $\theta^{T}_{F,S} \in \ZZ[\Gal(F/k)]$.

We put $S_{{\rm r}}=S_{{\rm ram}}(K/k)$.
For a subset $J$ of $S_{{\rm r}}$ we define $K_J$ to be the maximal 
subextension of $k$ in $K$ that are unramified at all primes in $J$. 
Namely $K_{J}$ is the fixed subfield of the subgroup of $G$ 
generated by $I_{v}$ 
for all $v \in J$. 
If $J$ is empty, we take $K_{J}=K$. 
We put $N_{J}=\prod_{v \in J} N_{I_{v}} \in \ZZ[G]$. 
Then the multiplication by $N_{J}$ defines a homomorphism 
$$\nu_{J}: \ZZ[\Gal(K_{J}/k)] \longrightarrow \ZZ[G].$$
Note that this is not a norm homomorphism for $K/K_{J}$ but 
the multiplication by some constant of the norm homomorphism.
We have 
$$\nu_{J}(\theta_{F_{J}, S_{{\rm r}} \setminus J}^{T})
=
\prod_{v \in J} N_{I_{v}} \prod_{v \in S_{{\rm r}} \setminus J}
(1-\frac{N_{I_{v}}}{\#I_{v}} {\mathcal F}_{v}^{-1}) \omega^{T}.
$$
This equality can be proved by comparing the $\chi$-components of 
both sides for 
each character $\chi$ of $G$ (see, for example, 
Lemma 2.1 in \cite{Ku1}).

This equality shows that $\Theta^{T}(K)$ is generated by 
$\nu_{J}(\theta_{F_{J}, S_{{\rm r}} \setminus J}^{T})$ for all 
subsets $J$ of $S_{{\rm r}}$.
In particular, we obtain $\Theta^{T}(K) \subset \ZZ[G]$.
This completes the proof.
\end{proof}

\vspace{5mm}

For any group ring $R[G]$ we denote by $x \mapsto x^{\#}$ the 
involution $R[G] \longrightarrow R[G]$ 
induced by $\sigma \mapsto \sigma^{-1}$ for 
all $\sigma \in G$. 
 
\begin{Conjecture} \label{C1}
Put $R=\ZZ[1/2][G]^-$, 
$((Cl_{K}^{T})')^{\vee}=((Cl_{K}^{T}\otimes \ZZ[1/2])^{-})^{\vee}$, 
and $\Theta^{T}(K)'=(\Theta^{T}(K) \otimes \ZZ[1/2])^{-} 
\subset R$. 
Then 
$$\Fitt_{R}(((Cl_{K}^{T})')^{\vee})=(\Theta^{T}(K)')^{\#}$$
holds true. 
\end{Conjecture}

Now we study this conjecture, using Proposition \ref{P2}.
Consider the $\QQ_{p}[G]^-$-homomorphism
$$
\psi^{\circ}:(W_{S_{\infty}}^{\circ} \otimes \QQ_{p})^- 
\longrightarrow (\GA^{\circ} \otimes \QQ_{p})^- \ .
$$ 
For a finite prime $v$ in $S'$ let $e_{v}$ be 
as in (\ref{BasisIsom3}) (see also (\ref{BasisIsom2})).
For an infinite prime $v$ we also defined 
$e_{v}$ of $W_{S_{\infty}}^{\circ} \otimes \QQ$
in \S \ref{CM-extensionsSubsection}.
We also write $e_{v}$ for the minus component of $e_{v}$, 
and take a basis $(e_{v})_{v \in S'}$ of 
$(W_{S_{\infty}}^{\circ} \otimes \QQ_{p})^-$.

We consider $\det \psi^{\circ} \in \QQ[G]^-$ 
with respect to the basis $(e_{v})_{v \in S'}$ 
and a basis of $(\GA^{\circ} \otimes \ZZ_{p})^-$
which is a free $\ZZ_{p}[G]^-$-module of rank $\#S'$. 
Then $\det \psi^{\circ}$ is 
determined up to unit of $\ZZ_{p}[G]^-$, 
and is a nonzero divisor of $\QQ_{p}[G]^-$.

Recall that $h_{v} \in \QQ[G]$ was defined in (\ref{hv}) 
(see also (\ref{gv})).

\begin{Theorem} \label{T1}
For any odd prime number $p$ we have 
$$\Fitt_{\ZZ_{p}[G]^-}(((Cl_{K}^{T})^{\vee})^-_{\ZZ_{p}})
=((\prod_{v \in S_{{\rm ram}}(K/k)} 
U_{v}^{\#}) (\prod_{v \in S' \setminus S_{\infty}} 
h_{v}^{-1}))_{\ZZ_{p}}^{-} 
\det \psi^{\circ}$$
where $\det \psi^{\circ}$ is taken with respect to 
$(e_{v})_{v \in S'}$ and a basis of 
$(\GA^{\circ} \otimes \ZZ_{p})^-$.
\end{Theorem}

\begin{proof}
We use the presentation of $((Cl_{K}^{T})^{\vee})^-_{\ZZ_{p}}$
in Proposition \ref{P2}. 
For a finite prime $v \in S'$ we proved in (\ref{BasisIsom4}) that 
$$W_{v}^{\circ}=h_{v}^{-1} U_{v}^{\#} \ZZ[G]e_{v}.$$
Therefore, the minus part of $(W_{S_{\infty}}^{\circ})_{\ZZ_{p}}$ 
can be written as 
$$(W_{S_{\infty}}^{\circ})^-_{\ZZ_{p}}  
=
\bigoplus_{v \in S_{\infty}} 
\ZZ_{p}[G]^- e_{v} \oplus 
\bigoplus_{v \in S' \setminus S_{\infty}} 
(h_{v}^{-1}U_{v}^{\#})_{\ZZ_{p}}^- e_{v}.
$$
It follow from the exact sequence (\ref{e4}) that 
$$\Fitt_{\ZZ_{p}[G]^-}(((Cl_{K}^{T})^{\vee})^-_{\ZZ_{p}})
=(\prod_{v \in S' \setminus S_{\infty}} 
h_{v}^{-1}U_{v}^{\#})_{\ZZ_{p}}^{-} 
\det \psi^{\circ}$$
If $v$ is unramified, we have $U_{v}=\ZZ[G]$, which 
implies the conclusion of Theorem \ref{T1}.
\end{proof}

By Theorem \ref{T1}, 
we know that Conjecture \ref{C1} is equivalent to 
\begin{equation} \label{Conjecture2}
(\prod_{v \in S' \setminus S_{\infty}} h_{v}^{-1})
\det \psi^{\circ} \cdot \ZZ_{p}[G]^{-}=
(\omega^T\ZZ_{p}[G]^{-})^{\#}
\end{equation}
for all odd $p$. 

\subsection{A conjecture on $\det \psi_{S}$} \label{ComparisonofTwoHom}

For a finite set $S$ such that 
$S_{\infty} \cup S_{{\rm ram}}(K/k) \subset S \subset S'$, 
we consider the homomorphism 
$\psi_{S}: \GA \longrightarrow \GB$ which was constructed in 
\S \ref{ConstructionofpsiS}, and study its determinant 
$\det \psi_{S}$. 

Since we defined $\GB$ by $\GB= \bigoplus_{w \in (S')_{K}} \ZZ[G_{w}]$, 
fixing a prime $w$ above $v$, we have 
$\GB=\bigoplus_{v \in S'} \ZZ[G]$. 
For each $v$, we take a canonical 
basis $(e_{v}^{\GB})_{v \in S'}$ of $\GB$ where 
$e^{\GB}_{v}$ is the element whose $v$-component is $1$ and 
other components are zero. 

For an intermediate field $F$ of $K/k$, let 
$\psi_{F,S}: \GA_{F} \longrightarrow \GB_{F}$ denote 
the $\psi_{S}$ for $F$. 
We write $\GA_{K}$, $\GB_{K}$ for $\GA$, $\GB$ in order to 
clarify the field over which these modules are defined. 
For modules $W_{S}$,$A$, $B$,..., we write $W_{F,S}$, $A_{F}$,
$B_{F}$,...for the corresponding modules for $F$.

%Suppose that $\GA_{K}$ is free over $\ZZ[G]$.
%Taking the $\Gal(K/F)$-coinvariants of $\GA_{K}$, $\GB_{K}$, 
%we get isomorphisms $(\GA_{K})_{\Gal(K/F)} \simeq 
%\GA_{F}$, $(\GB_{K})_{\Gal(K/F)} \simeq \GB_{F}$.
%So a basis of $\GA_{K}$ as a $\ZZ[G]$-module 
%gives a basis of $\GA_{F}$ as a $\ZZ[\Gal(F/k)]$-module. 
%For $\GB_{F}$ we always use the canonical basis which 
%we explained above. 
%So $(e_{F,v}^{\GB})_{v \in S'}$ is also 
%the image of $(e_{K,v}^{\GB})_{v \in S'}$ under the canonical 
%map $\GB_{K} \longrightarrow \GB_{F}$.

In order to compare $\psi_{F, S}$ for several $F$ and $S$ below, 
it is convenient to remove the ambiguity of the definition of 
this map (recall that $\psi_{F,S}$ was defined as an extension of 
$\psi_{F,S}: A_{F} \longrightarrow B_{F}$). 
We take and fix an infinite prime $v_{\infty} \in S$, and 
define $\psi_{F,S}: \GA_{F}=A_{F} \oplus \ZZ[\Gal(F/k)] 
\longrightarrow \GB_{F}$ by 
$\psi_{F,S}((0,1))=e_{F,v_{\infty}}^{\GB}$.

For $S$ such that $S_{\infty} \cup S_{{\rm ram}}(F/k) 
\subset S \subset S'$, we define 
$\theta_{F,S}^{T} \in \ZZ[\Gal(F/k)]$ as in (\ref{ThetaST}).

\begin{Conjecture} \label{C2}
The module $\GA_{K}$ is a free $\ZZ[G]$-module with 
a basis $(e_{K,v}^{\GA})_{v \in S'}$ such that 
for any intermediate field $F$ of $K/k$ and 
for any $S$ such that $S_{\infty} \cup S_{{\rm ram}}(F/k) 
\subset S \subset S'$, we have 
$$
\det(\psi_{F,S})=\theta_{F,S}^{T}
$$
Here, we define a basis $(e_{F,v}^{\GA})_{v \in S'}$ of $\GA_{F}$
as the image of $(e_{K,v}^{\GA})_{v \in S'}$
under the natural map $\GA_{K} \longrightarrow \GA_{F}$, 
and a basis $(e_{F,v}^{\GB})_{v \in S'}$ of $\GB_{F}$ 
as the canonical basis for $\GB_{F}$ (so also the image of 
$(e_{K,v}^{\GB})_{v \in S'}$), and 
$\det(\psi_{F,S})$ is taken with respect to the bases
$(e_{F,v}^{\GA})_{v \in S'}$ of $\GA_{F}$ and  
$(e_{F,v}^{\GB})_{v \in S'}$ of $\GB_{F}$.
\end{Conjecture}

We note that $\det (\psi_{F,S})=\theta_{F,S}^{T}$ in 
Conjecture \ref{C2} is the equality of not ideals, 
but of elements in $\ZZ[\Gal(F/k)]$.
Also, this conjecture asserts the existence of a good 
basis which can be used for any $F$ and $S$. 
This equivariant statement would remind one of  
the equivariant Tamagawa number conjecture. 
In fact, 

\begin{Proposition} \label{eTNCimpliesC2}
The equivariant Tamagawa number conjecture for $K/k$
(eTNC in short) implies Conjecture \ref{C2}.
\end{Proposition}

\begin{proof}
We assume that eTNC holds true. 
We use Conjecture 3.6 in \cite{BKS1} as eTNC, which claims that 
there is an element $z_{K/k,S,T}$ which is a basis of  
$\det_{G}R\Gamma_{T}(\OO_{K,S}, {\mathbb G}_{m})$ as a 
$\ZZ[G]$-module 
such that 
$\vartheta_{\lambda_{K,S}}(z_{K/k,S,T})=\theta^{{T \ {*}}}_{K/k,S}(0)$
where 
$$
\vartheta_{\lambda_{K,S}}:
{\rm det}_{G}R\Gamma_{T}(\OO_{K,S}, {\mathbb G}_{m}) \otimes {\mathbb R}
\stackrel{\simeq}{\longrightarrow} {\mathbb R}[G]
$$
is the isomorphism 
defined by using the Dirichlet regulator,  
and $\theta^{T \ {*}}_{K/k,S}(0)$ is the leading term of 
$(S,T)$-modified equivariant zeta function 
$\theta^T_{K/k,S}(s)$ at $s=0$ 
(see \S 3 in \cite{BKS1}).
Since the complex 
$R\Gamma_{T}(\OO_{K,S}, {\mathbb G}_{m})$ is 
represented by $\GA_{K} \stackrel{\psi_{S}}{\longrightarrow} \GB_{K}$, 
$\GA_{K}$ is free. 
Also, since we fixed a basis of $\GB_{K}$,  
$z_{K/k,S,T}$ yields a basis of $\GA_{K}$
which we denote by $(e_{K,v}^{\GA})_{v \in S'}$, and which 
we use from now on. 
By definition, we have 
$
\det(\psi_{K,S})=\theta^T_{K/k,S}(0)=\theta^{T}_{K/k,S}.
$
Also, for an intermediate field $F$, we know that the zeta element 
$z_{F/k,S,T}$ is the image of $z_{K/k,S,T}$.
This shows that 
$$
\det(\psi_{F,S})=\theta^T_{F/k,S}(0)=\theta^{T}_{F/k,S}.
$$

Suppose that $v$ is in $S \setminus S_{\rm ram}(K/k)$, and put 
$S''=S \setminus \{v\}$. 
We will next prove 
$\det(\psi_{F,S''})=\theta^{T}_{F/k,S''}$.

We first suppose that $v$ splits completely in $F$. 
We consider the natural homomorphism 
$i_{S'',S}:R\Gamma_{T}(\OO_{F,S''}, {\mathbb G}_{m}) 
\longrightarrow 
R\Gamma_{T}(\OO_{F,S}, {\mathbb G}_{m})$.
Then we know 
$$
\vartheta_{\lambda_{F,S}}(i_{S'',S}(z_{F/k,S'',T}))
=(\log N(v)) \theta^{T \ {*}}_{F/k,S''}(0)
=\theta^{T \ {*}}_{F/k,S}(0)
$$
(see, for example, 
the proof of \cite{BKS1} Proposition 3.4).
Therefore, we get 
$i_{S'',S}(z_{K/k,S'',T})=z_{K/k,S,T}$. 
This implies that  
$\det(\psi_{F,S''})=\theta^{T}_{F/k,S''}$.

Next, we consider a general $v$. 
For an element $x \in \QQ[\Gal(F/k)]$ and 
a character $\chi$ of $\Gal(F/k)$, 
we denote by $\epsilon_{\chi}=\epsilon_{F,\chi}$ 
the idempotent of the $\chi$-component for $\Gal(F/k)$, and
%as in the previous subsection, and 
write $x^{\chi}=\epsilon_{\chi}x$ which is an element 
of the $\chi$-component of 
$\QQ(\mu_{m})[\Gal(F/k)]$ where $m=\# \Gal(F/k)$. 
In order to prove $\det(\psi_{F,S''})=\theta^{T}_{F/k,S''}$,
it suffices to show the equality 
$\det(\psi_{F,S''})^{\chi}=(\theta^{T}_{F/k,S''})^{\chi}$
for all characters $\chi$ of 
$\Gal(F/k)$.

Note that $v$ is unramified in $F$. 
If $\chi({\mathcal F}_{v})=1$, then we can prove this equality 
by the method in the previous paragraph. 
So we assume $\chi({\mathcal F}_{v}) \neq 1$.
 
The images of $\psi_{F,S}$, $\psi_{F,S''}$ are in $W_{F,S}$, 
$W_{F,S''}$, respectively. 
The difference between $W_{F,S}$ and $W_{F,S''}$ lies only on 
the $v$-component; the former is 
$\Delta_{F,v}=\bigoplus_{w \vert v} \Delta G_{w}(F/k)$ and 
the latter is $W_{F,v} \simeq \ZZ[\Gal(F/k)]$ which is defined by 
$(x,y) \mapsto y$.
If $(x,y)$ is in $W_{F,v}$, then $x=(1- {\mathcal F}_{v}^{-1})y$
by definition. 
Therefore, the natural map $W_{F,S''} \longrightarrow W_{F,S}$ 
is the multiplication by $1-{\mathcal F}_{v}^{-1}$ on the 
$v$-component and the identity on other components. 
Let $\phi_{v}: \GB_{F} \longrightarrow \GB_{F}$  
be the map which is the multiplication by 
$1-{\mathcal F}_{v}^{-1}$ on the $v$-component and 
the identity on other components.
Then we have 
$$\psi_{F,S}= \phi_{v} \circ \psi_{F,S''}.$$
Since $\det \phi_{v}=1-{\mathcal F}_{v}^{-1}$, we get 
$$\det(\psi_{F,S})^{\chi}=(1-\chi({\mathcal F}_{v})^{-1})
\det(\psi_{F,S''})^{\chi}.$$ 
Therefore, the equality 
$\det(\psi_{F,S})^{\chi}=(\theta^{T}_{F/k,S})^{\chi}$ we obtained 
above implies 
$\det(\psi_{F,S''})^{\chi}=(\theta^{T}_{F/k,S''})^{\chi}$.
Now we have obtained the equality for all $\chi$-components, so 
we get 
$$\det(\psi_{F,S''})=\theta^{T}_{F/k,S''}.$$

By induction on $\#(S' \setminus S)$, 
starting from $S=S'$ and applying the above argument, 
we obtain for any $S$ and any $F$ 
%we get 
$$\det(\psi_{F,S})=\theta^{T}_{F/k,S}.$$
%$for any $S$ and $F$. 
\end{proof}

It is also easily checked by the argument in the above proof that 
Conjecture \ref{C2} implies the eTNC, namely the existence of 
$z_{K/k,S,T}$.

\vspace{5mm}

We assume Conjecture \ref{C2}, so the existence of 
a basis $(e_{K,v}^{\GA})_{v \in S'}$ of $\GA_{K}$. 
We denote by $(e_{K,v}^{\GA^\circ})_{v \in S'}$ 
the dual basis of $\GA_{K}^{\circ}$. 
We next study the homomorphism
$$
\psi^{\circ}:(W_{K, S_{\infty}}^{\circ} \otimes \QQ_{p})^- 
\longrightarrow (\GA_{K}^{\circ} \otimes \QQ_{p})^- \ .
$$ 
in Proposition \ref{P2}.
We take $(e_{v})_{v \in S'}$ as a basis of 
$W_{S_{\infty}}^{\circ}$ 
as in Theorem \ref{T1}, and $(e_{K,v}^{\GA^{\circ}})_{v \in S'}$
as a basis of $(\GA_{K}^{\circ} \otimes \QQ_{p})^-$ to study
$\det \psi^{\circ} \in \QQ_{p}[G]^-$.

\begin{Theorem} \label{theorem2}
We assume Conjecture \ref{C2}.\\
{\rm (1)} We have 
$$
\det \psi^{\circ} = (\omega^T)^{\#} 
\prod_{v \in S' \setminus S_{\infty}} h_{v}
$$
where $\det \psi^{\circ}$ is taken with respect to the bases 
$(e_{v})_{v \in S'}$ and $(e_{K,v}^{\GA^{\circ}})_{v \in S'}$, 
and $h_{v}$ was defined in (\ref{hv}).\\
{\rm (2)} Conjecture \ref{C1} holds, namely
$$
\Fitt_{\ZZ[1/2]}(((Cl_{K}^{T})')^{\vee})=(\Theta^{T}(K)')^{\#}.
$$
\end{Theorem}

\begin{proof} Theorem \ref{theorem2} (2) is a consequence of 
Theorems \ref{T1} and \ref{theorem2} (1) 
(see also (\ref{Conjecture2})). 
So it suffices to prove Theorem \ref{theorem2} (1).
To do this, we prove 
$$
(\det \psi^{\circ})^{\chi} = ((\omega^T)^{\#} 
\prod_{v \in S' \setminus S_{\infty}} h_{v})^{\chi}
=L_{T}(0,\chi)\prod_{v \in S' \setminus S_{\infty}} h_{v}^{\chi}
$$
for any character $\chi$ of $G$ 
where we denote the $\chi$-component $\epsilon_{\chi}x$ 
by $x^{\chi}$ for any element $x$ in $\QQ_{p}[G]^-$.

We use the notation in \S \ref{CM-extensionsSubsection}.
Suppose that $v$ is a finite prime in $S'$, and $w$ is the 
prime we fixed above $v$.
It follows from (\ref{BasisIsom2}) that 
$c_{w}((0,1))=g_{v}^{-1}N_{I_{v}}c_{w}((1,1))$ and
\begin{equation} \label{cw10}
c_{w}((1,0))=c_{w}((1,1))-c_{w}((0,1))=
(1-g_{v}^{-1}N_{I_{v}})c_{w}((1,1)).
\end{equation}

Let $K_{\chi}/k$ be the intermediate field of $K/k$ corresponding 
to $\Ker \chi$.
We put $F=K_{\chi}$, $S_{\chi}=S_{\infty} \cup 
S_{{\rm ram}}(F/k)=S_{\infty} \cup 
S_{{\rm ram}}(K_{\chi}/k)$, 
and consider 
$\psi_{F, S_{\chi}}:\GA_{F} \longrightarrow \GB_{F}$ and 
its dual $\psi_{F, S_{\chi}}^{\circ}:
\GB_{F}^{\circ} \longrightarrow \GA_{F}^{\circ}$
with which we compare 
$\psi^{\circ} 
:(W_{K, S_{\infty}}^{\circ} \otimes \QQ_{p})^- 
\longrightarrow (\GA_{K}^{\circ} \otimes \QQ_{p})^-$. 
We compute the image of the natural homomorphism 
$$\iota: \GB_{F}^{\circ} \stackrel{\alpha}{\longrightarrow} 
W_{F,S_{\infty}}^{\circ} \otimes \QQ
\stackrel{\beta}{\longrightarrow} 
W_{K,S_{\infty}}^{\circ} \otimes \QQ$$
with the dual basis $(e_{K,v}^{\GB^\circ})_{v \in S'}$ of  
$\GB_{F}^{\circ}$, obtained from 
$(e_{F,v}^{\GB})_{v \in S'}$ and 
the basis $(e_{v})_{v \in S'}$ of 
$W_{K,S_{\infty}}^{\circ} \otimes \QQ$.
(Note that since $\GB_{F}$ was constructed from $W_{F,S}$, 
it depends on $S$ though the notation does not carry $S$. 
In our case above, $S=S_{\chi}$.)

We denote by $w'$ the prime of $F$ below $w$
that is the prime we fixed of $K$ above $v$. 
Put $H_{w}=\Gal(K_{w}/F_{w'})$, and 
$G_{w'}=\Gal(F_{w'}/k)$, then since 
$G_{w}=\Gal(K_{w}/k_{v})$, we have 
$G_{w'}=G_{w}/H_{w}$. 

Suppose at first $v$ is in $S_{\chi}=S_{{\rm ram}}(F/k)$.
We will prove 
\begin{equation} \label{equationramifiedcase}
\iota(e_{F,v}^{\GB^\circ})
=N_{H_{w}}(1-g_{v}^{-1}N_{I_{v}})e_{v}.
\end{equation}
Since $w'$ is ramified, the $w'$-component of
the natural map $W_{F,S_{\chi}} \longrightarrow \GB_{F}$ is 
$W_{F_{w'}} \longrightarrow \ZZ[G_{w}]; (x,y) \mapsto x$. 
Let 
$$c_{w'}: \QQ[G_{w'}] \oplus \QQ[G_{w'}] 
\longrightarrow W_{F_{w'}}^{\circ} \otimes \QQ$$
be the homomorphism  
obtained by applying the definition of $c_{w}$ 
in \S \ref{CM-extensionsSubsection} 
to $w'$.
We consider the natural map 
$W_{F,S_{\infty}} \longrightarrow \GB_{F}$
and its dual $\alpha: \GB_{F}^{\circ} \longrightarrow 
W_{F,S_{\infty}}^{\circ} 
\subset W_{F,S_{\infty}}^{\circ} \otimes \QQ$.
Then the $w'$-component $\alpha_{w'}$ of $\alpha$,
%is  
%Then the $w'$-component of $\GB_{F}^{\circ}$ 
%$\ZZ[G_{w'}]$ and 
$$\alpha_{w'}: \ZZ[G_{w'}] \longrightarrow 
W_{F_{w'}}^{\circ} \otimes \QQ$$ 
is described as $\alpha_{w'}(1)= c_{w'}((1,0))$ 
by what we explained above and the definitions 
of the modules. 
Since the diagram 
$$
\begin{array}{ccc}
\ZZ[G_{w'}] & \stackrel{\alpha_{w'}}{\longrightarrow} &
W_{F_{w'}}^{\circ} \otimes \QQ \\
\mapdown{N_{H_{w}}}&&\mapdown{}\\
\ZZ[G_{w}] & \stackrel{p_1}{\longrightarrow} &
W_{K_{w}}^{\circ} \otimes \QQ \\
\end{array}
$$
is commutative where the bottom map $p_{1}$ is 
$p_{1}(x)=c_{w}((x,0))$, 
the $w'$-component of 
$\iota= \beta \circ \alpha$ can be described as 
$$\ZZ[G_{w'}] \longrightarrow 
W_{F_{w}}^{\circ} \otimes \QQ; \ 
1 \mapsto N_{H_{w}}c_{w}((1,0))=
N_{H_{w}}(1-g_{v}^{-1}N_{I_{v}})c_{w}((1,1))$$
where we used (\ref{cw10}) to get the last equality.
This shows that  
$$\iota(e_{F,v}^{\GB^\circ})
=N_{H_{w}}(1-g_{v}^{-1}N_{I_{v}})e_{v},$$
which completes the proof of (\ref{equationramifiedcase}).

%$$\iota(e_{F,v}^{\GB^\circ})
%=N_{H_{w}}(1-g_{v}^{-1}N_{I_{v}})e_{v}.$$
Since $v$ is ramified, taking the $\chi$-component 
(multiplying (\ref{equationramifiedcase}) by 
$\epsilon_{\chi}$), we get 
\begin{equation} \label{ComparisonBasis1}
\iota(e_{F,v}^{\GB^\circ} \epsilon_{F, \chi})
=e_{v} \epsilon_{\chi}
\end{equation}
where $\epsilon_{F, \chi}=\#\Gal(F/k)^{-1}
\sum_{\sigma \in \Gal(F/k)} \chi(\sigma) \sigma^{-1}$ is the 
idempotent of the $\chi$-component of 
the group ring for $\Gal(F/k)$.

\vspace{5mm}

Next, suppose that $v$ is unramified in $F=K_{\chi}$.
This time $v$ is not in $S_{\chi}$, so 
the $w'$-component of
$W_{F,S_{\chi}} \longrightarrow \GB_{F}$ is 
$W_{F_{w'}} \longrightarrow \ZZ[G_{w}]; (x,y) \mapsto y$.
Therefore, $\alpha_{w'}: \ZZ[G_{w'}] \longrightarrow 
W_{F_{w'}}^{\circ} \otimes \QQ$ 
is described as 
$$\alpha_{w'}(1)= c_{w'}((0,1)).$$ 
Since $v$ is unramified, $I_{v}$ is in $H_{w}$. 
We note that the map $x \mapsto c_{w}((0,1))$ 
factors through $\ZZ[G_{w}/I_{w}]$.
We denote this map 
$\ZZ[G_{w}/I_{w}] \longrightarrow W_{K_{w}}^{\circ} \otimes \QQ$
by $p_{2}$. 
Then the diagram
$$
\begin{array}{ccc}
\ZZ[G_{w'}] & \stackrel{\alpha_{w'}}{\longrightarrow} &
W_{F_{w'}}^{\circ} \otimes \QQ \\
\mapdown{N_{H_{w}/I_{v}}}&&\mapdown{}\\
\ZZ[G_{w}/I_{v}] & \stackrel{p_{2}}{\longrightarrow} &
W_{K_{w}}^{\circ} \otimes \QQ \\
\end{array}
$$
is commutative. 
Thus the $w'$-component of 
$\iota=\beta \circ \alpha$, 
$\ZZ[G_{w'}] \longrightarrow 
W_{F_{w}}^{\circ} \otimes \QQ$
is described as 
$$1 \mapsto N_{H_{w}/I_{v}}c_{w}((0,1))=
N_{H_{w}/I_{v}}N_{I_{v}}g_{v}^{-1}c_{w}((1,1))
=N_{H_{w}}g_{v}^{-1}c_{w}((1,1))$$
where we used (\ref{BasisIsom2}) to get the first equality.
This implies that 
\begin{equation} \label{equationunramifiedcase}
\iota(e_{F,v}^{\GB^\circ})
=N_{H_{w}}g_{v}^{-1}e_{v}.
\end{equation}
Multiplying $\epsilon_{\chi}$, we now get 
\begin{equation} \label{ComparisonBasis2}
\iota(e_{F,v}^{\GB^\circ} \epsilon_{F, \chi})
=g_{v}^{-1}e_{v} \epsilon_{\chi}.
\end{equation}

Recall that $\det \psi$, $\det \psi_{F,S_{\chi}}$ are 
computed by using the basis 
$(e_{v})_{v \in S'}$, $(e_{F,v}^{\GB^\circ})_{v \in S'}$, 
respectively.
Therefore, it follows from (\ref{ComparisonBasis1}) 
and (\ref{ComparisonBasis2}) that 
$$
\det(\psi_{F,S_{\chi}}^{\circ})^{\chi}=(
(\prod_{S' \setminus (S_{\infty} \cup S_{\chi})} g_{v}^{-1})
\det(\psi^{\circ}))^{\chi}
=((\prod_{S' \setminus S_{\infty}} h_{v}^{-1})
\det(\psi^{\circ}))^{\chi}.
$$
To get the last equality, we used (\ref{hv}).
Using Conjecture \ref{C2}, we obtain
\begin{eqnarray*}
\det(\psi^{\circ})^{\chi} &= &
((\prod_{S' \setminus S_{\infty}} h_{v}) 
\det(\psi_{F,S_{\chi}}^{\circ}))^{\chi}
= 
((\prod_{S' \setminus S_{\infty}} h_{v}) 
(\theta_{F, S_{\chi}}^{T})^{\#})^{\chi} \\
&= &L_{T}(0,\chi)\prod_{v \in S' \setminus S_{\infty}} h_{v}^{\chi}.
\end{eqnarray*}
This holds for all characters $\chi$ of $G$, 
so we get the desired equality in Theorem \ref{theorem2} (1). 
\end{proof}

\begin{Corollary} \label{MainCorollary}
The equivariant Tamagawa number conjecture for $K/k$ implies 
Conjecture \ref{C1}.
\end{Corollary}

\begin{proof}
This follows from Theorem \ref{theorem2} (2) and 
Proposition \ref{eTNCimpliesC2}.
\end{proof}

%This can be proved by the same argument as the main theorem of 
%\cite{Grei4}. 
%We note that the module $W^{\circ}$ is independent of $T$. 

\section{Cyclotomic $\ZZ_{p}$-extensions} \label{sec4}

Let $K_{\infty}/K$ be the cyclotomic $\ZZ_{p}$-extension and 
$K_{n}$ the $n$-th layer. 
Put $\Lambda=\ZZ_{p}[[\Gal(K_{\infty}/k)]]$. 
We first take the projective limit of the sequence (\ref{e4}) in Proposition \ref{P2}. 

We denote by $S_{\rm ram}=S_{\rm ram}(K_{\infty}/k)$ 
the set of all finite primes of $k$ 
ramifying in $K_{\infty}$. 
The set $S_{p}$ of all primes above $p$ is contained in $S_{\rm ram}$. 
We put $S^{{\rm non} \ p}_{\rm ram}=S_{\rm ram} \setminus S_{p}$.  
We take $S'$ which satisfies the conditions in 
\S \ref{ConstructionofpsiS} 
for $K/k$ and which satisfies $S' \supset S_{\rm ram}$.

We consider $W_{K_{n}, S_{\infty}}$ which is $W_{S_{\infty}}$
in \S \ref{CM-extensionsSubsection} for $K_{n}$. 
Let $w$ be a prime of $K_{\infty}$. 
We also denote by $w$ the prime of $K_{n}$ below $w$ and 
consider $W_{K_{n,w}}$.
We define 
$$
W(K_{\infty}/k)^{\circ}_{\ZZ_{p}} = {\lim\limits_{\leftarrow}} 
(W_{K_{n},S_{\infty}}^\circ \otimes \ZZ_{p}),
$$
$$
W_{w}(K_{\infty}/k)^{\circ}_{\ZZ_{p}} = {\lim\limits_{\leftarrow}} 
(W_{K_{n,w}}^\circ \otimes \ZZ_{p})
$$
and 
$W_{v}(K_{\infty}/k)^\circ_{\ZZ_{p}}=\bigoplus_{w \mid v} 
W_{w}(K_{\infty}/k)^\circ_{\ZZ_{p}}$ 
for a finite prime $v$ of $k$.

We first consider a prime $w$ above $p$.
Suppose that $n$ is sufficiently large such that 
$K_{\infty}/K_{n}$ is totally ramified at all primes above $p$.
Consider the canonical exact sequences for $W_{K_{n,w}}$ 
and $W_{K_{n+1},w}$ (see (1.4) in \cite{GruW}). 
Then we have a commutative diagram of exact sequences
$$
\begin{array}{ccccccccc}
0 & \longrightarrow & \ZZ & 
\longrightarrow & W_{K_{n,w}} & \longrightarrow & \Delta D_{w}(K_{n}/k)
& \longrightarrow & 0 \\
&& \mapdown{\xi} && \mapdown{} && \mapdown{\nu} && \\
0 & \longrightarrow & \ZZ & 
\longrightarrow & W_{K_{n+1,w}} & \longrightarrow & \Delta D_{w}(K_{n+1}/k)
& \longrightarrow & 0 \\
\end{array}
$$
where $D_{w}(K_{n}/k)$ is the decomposition subgroup of $w$ 
in $\Gal(K_{n}/k)$, 
$\xi$ is the multiplication by $p$ and $\nu$ is the norm map.
The above commutative diagram shows that for a prime $w$ of $K_{\infty}$ above $p$, 
$$W_{w}(K_{\infty}/k)^{\circ}_{\ZZ_{p}}={\lim\limits_{\leftarrow}} 
\ZZ_{p}[D_{w}(K_{n}/k)]/(N_{D_{w}(K_{n}/k)}) 
= \ZZ_{p}[[D_{w}(K_{\infty}/k)]].$$
Therefore, $W_{v}(K_{\infty}/k)^\circ_{\ZZ_{p}}
=\bigoplus_{w \mid v} 
W_{w}(K_{\infty}/k)^\circ_{\ZZ_{p}}$ 
is a free $\Lambda$-module 
of rank $1$ for $v \in S_{p}$. 

We use the notation in \S \ref{CM-extensionsSubsection}.
Let $c_{w}:\ZZ[G_{w}(K_{n}/k)] \oplus \ZZ[G_{w}(K_{n}/k)]
\longrightarrow W_{K_{n},w}^{\circ}$ 
be the map obtained by applying to $K_{n}/k$  
the definition for $K/k$ before Proposition \ref{P2} in 
\S \ref{CM-extensionsSubsection}.
Taking the projective limit, we have a map 
$$c_{w}: \ZZ_{p}[[D_{w}(K_{\infty}/k)]] \oplus 
\ZZ_{p}[[D_{w}(K_{\infty}/k)]] \longrightarrow 
W_{w}(K_{\infty}/k)^{\circ}_{\ZZ_{p}}.$$
What we have shown in the previous paragraph, 
means that $c_{w}((1,0))$ generates 
$W_{w}(K_{\infty}/k)^{\circ}_{\ZZ_{p}}$,
namely 
$W_{w}(K_{\infty}/k)^{\circ}_{\ZZ_{p}}=
\ZZ_{p}[[D_{w}(K_{\infty}/k)]] c_{w}((1,0))$.

Fixing $w$ above $v$, we have a map
$$c_{v}: \ZZ_{p}[[\Gal(K_{\infty}/k)]] \oplus 
\ZZ_{p}[[\Gal(K_{\infty}/k)]] = \Lambda \oplus \Lambda 
\longrightarrow 
W_{v}(K_{\infty}/k)^{\circ}_{\ZZ_{p}}.$$
We put $e_{v}'=c_{v}((1,0))$. 
Then we have 
$W_{v}(K_{\infty}/k)^{\circ}_{\ZZ_{p}}=\Lambda e_{v}'$.

\vspace{5mm}

Next, suppose that $v$ is a non $p$-adic finite prime.
Note that the inertia group $I_{v}(K_{\infty}/k)$ of 
$\Gal(K_{\infty}/k)$ coincides with the inertia group
$I_{v}(K/k)$ of $\Gal(K/k)$. 
We denote it by $I_{v}$. 

We define $c_{v}$ as above and also define 
$e_{v}'=c_{v}((1,0))$.
Then the map $\Lambda \longrightarrow 
W_{v}(K_{\infty}/k)^{\circ}_{\ZZ_{p}}$, 
$a \mapsto ae_{v}'$ is injective because
${\mathcal F}_{v}-1$ is a nonzero divisor in $\Lambda$. 

Let ${\mathcal R}$ be the total quotient ring of $\Lambda$. 
Then ${\mathcal R} \longrightarrow 
W_{v}(K_{\infty}/k)^{\circ}_{\ZZ_{p}}\otimes {\mathcal R}$ 
which is defined by $a \mapsto ae_{v}'$ 
is bijective.
Since $W_{v}(K_{\infty}/k)^{\circ}_{\ZZ_{p}}$ is generated by 
$c_{v}((1,0))$ and $c_{v}((0,1))$, 
we have 
\begin{equation} \label{ev'isomorphism}
W_{v}(K_{\infty}/k)^{\circ}_{\ZZ_{p}}
=(1, \frac{N_{I_{v}}}{1-{\mathcal F}_{v}}) \Lambda e_{v}'.
\end{equation}

We now suppose that $v$ is an infinite prime. 
Then the $v$-component of $W_{K_{n}}$ is canonically 
isomorphic to 
$\ZZ[\Gal(K_{n}/k)]$ (after fixing a prime $w$ above $v$).  
We took a generator $e_{v}$ of $(W_{K_{n},v}^{\circ})_{\ZZ_{p}}^-$ 
in \S \ref{CM-extensionsSubsection}.
We define $e_{v}'$ to be the projective limit of $e_{v}$ 
as $n \rightarrow \infty$.  
So in this case we have 
$W_{v}(K_{\infty}/k)^{\circ}_{\ZZ_{p}}=\Lambda e_{v}'$.
Thus if $v$ is $p$-adic or infinite, 
$$W_{v}(K_{\infty}/k)^{\circ}_{\ZZ_{p}}=\Lambda e_{v}'$$
holds.

We regard $e_{v}'$ as an element of $W(K_{\infty}/k)^{\circ}_{\ZZ_{p}}$
(by defining that 
the $v'$-component of $e_{v}'$ is zero for all $v' \neq v$). 
Then $(e_{v}')_{v \in S'}$ is a basis of a free 
${\mathcal R}$-module 
$W(K_{\infty}/k)^{\circ}_{\ZZ_{p}} \otimes {\mathcal R}$.

\vspace{5mm}

We first note that 
$\Coker(V_{K_{n},S'}^{T} \longrightarrow \GO_{K_{n}}) 
\otimes \ZZ_{p}=0$ for any $n \geq 0$ 
where $V_{K_{n},S'}^{T}$, $\GO_{K_{n}}$ are 
$V_{S'}^{T}$, $\GO$ for $K_{n}$. 
This can be checked as follows. 
Put $G_{n}=\Gal(K_{n}/K)$. 
Since the natural maps induce isomorphisms 
$(V_{K_{n},S'}^{T})_{G_{n}} \simeq V_{K,S'}^{T}$ and 
$(\GO_{K_{n}})_{G_{n}} \simeq \GO_{K}$, 
the surjectivity of $V_{K,S'}^{T} \longrightarrow \GO_{K}$ 
implies $\Coker(V_{K_{n},S'}^{T} \longrightarrow \GO_{K_{n}})_{G_{n}}
=0$. 
Therefore, Nakayama's lemma implies 
$\Coker(V_{K_{n},S'}^{T} \longrightarrow \GO_{K_{n}}) 
\otimes \ZZ_{p}=0$.

Thus we have exact sequences (\ref{e4}) in Proposition \ref{P2} 
for any $K_{n}$, 
and can take the projective limit.

Consider $\GA^{\circ}_{K_{n}}$ 
which is $\GA^{\circ}$ for $K_{n}$, and define 
$$\GA^{\circ}(K_{\infty}/k)_{\ZZ_{p}} = {\lim\limits_{\leftarrow}} 
(\GA^{\circ}_{K_{n}} \otimes \ZZ_{p}).$$
The minus part $\GA^{\circ}(K_{\infty}/k)^-_{\ZZ_{p}}$ is a free 
$\Lambda^-$-module of finite rank.
We put 
$$Cl_{K_{\infty},p}^{T}={\lim\limits_{\rightarrow}} (Cl_{K_{n}}^{T}
\otimes \ZZ_{p}).$$
Taking the projective limit of the exact sequence (\ref{e4}), 
we have an exact sequence
\begin{equation} \label{e5}
0 \longrightarrow  (W(K_{\infty}/k)^{\circ}_{\ZZ_{p}})^-
\longrightarrow \GA^{\circ}(K_{\infty}/k)^-_{\ZZ_{p}} 
\longrightarrow ((Cl_{K_{\infty},p}^{T})^{\vee})^- \longrightarrow 0.
\end{equation}

Let $W'$ be the $\Lambda$-submodule of 
$(W(K_{\infty}/k)^{\circ}_{\ZZ_{p}})^-$
generated 
by $e_{v}'$ for all $v \in S'$. 
Then $W'$ is a free $\Lambda^-$-module.  
We write $f$ for the restriction of the homomorphism 
$(W(K_{\infty}/k)^{\circ}_{\ZZ_{p}})^-
\longrightarrow \GA^{\circ}(K_{\infty}/k)^-_{\ZZ_{p}}$
to $W'$. 
We consider $\det f$ with respect to the basis 
$(e_{v}')_{v \in S'}$. 
So $\det f$ is determined up to $\Lambda^{\times}$.

\begin{Lemma} \label{ML}
Suppose that $f:W' \longrightarrow 
\GA^{\circ}(K_{\infty}/k)^-_{\ZZ_{p}}$
is the homomorphism defined above, and 
we take $\det f$ with respect to the basis 
$(e_{v}')_{v \in S'}$. 
Then we have 
$$\Fitt_{\Lambda^-}(((Cl_{K_{\infty},p}^{T})^{\vee})^-)
= (\prod_{v \in S' \setminus (S_{\infty} \cup S_{p})}
(1, \frac{N_{I_{v}}}{1-{\mathcal F}_{v}})) \det f
$$
where $I_{v}=I_{v}(K/k)$ for each $v$. 
\end{Lemma}

\begin{proof}
If $v \in S_{\infty} \cup S_{p}$, we know 
$W_{v}(K_{\infty}/k)^{\circ}_{\ZZ_{p}}=\Lambda e_{v}'$.
For $v \in S' \setminus (S_{\infty} \cup S_{p})$, we have
$W_{v}(K_{\infty}/k)^{\circ}_{\ZZ_{p}}
=(1, \frac{N_{I_{v}}}{1-{\mathcal F}_{v}}) \Lambda e_{v}'$
by (\ref{ev'isomorphism}). 
Therefore, we have
$$(W(K_{\infty}/k)^{\circ}_{\ZZ_{p}})^-=
\bigoplus_{v \in S_{\infty} \cup S_{p}} \Lambda^- e_{v}'
\oplus
\bigoplus_{v \in S' \setminus (S_{\infty} \cup S_{p})}
(1, \frac{N_{I_{v}}}{1-{\mathcal F}_{v}}) \Lambda^- e_{v}'.
$$
Therefore, it follows from (\ref{e5}) that 
$$\Fitt_{\Lambda^-}(((Cl_{K_{\infty},p}^{T})^{\vee})^-)
= (\prod_{v \in S' \setminus (S_{\infty} \cup S_{p})}
(1, \frac{N_{I_{v}}}{1-{\mathcal F}_{v}})) \det f \ .
$$
\end{proof}

Our final task is to determine $\det f$. 

For a finite set $S$ which contains all ramifying primes in $K_{\infty}$, 
%as in the proof of Proposition \ref{P3}, 
we denote by $\theta_{K_{n},S}^{T}$ the 
$(S,T)$-modified Stickelberger element as in (\ref{ThetaST}), and 
by $\theta_{K_{\infty},S}^{T}$ its projective limit (for $n \gg 0$) 
in $\Lambda^-$.
We simply write $\theta_{K_{\infty}}^{T}$ when $S=S_{\rm ram}$. 
Also, for an intermediate CM-subfield $F$ of $K/k$
and the cyclotomic $\ZZ_{p}$-extension $F_{\infty}/F$, we define 
$\theta_{F_{\infty}}^{T}$ to be 
$\theta^{T}_{F_{\infty}, S_{\rm ram}(F_{\infty}/k)}$ where  
$S_{\rm ram}(F_{\infty}/k)$ is the set of all ramifying primes in 
$F_{\infty}/k$. 
We also use elements $\theta_{F_{\infty},S}^{T \hspace{0.5mm} \#}$, 
$\theta_{F_{\infty}}^{T \hspace{0.5mm} \#}$,... where 
$\#$ is the involution of the group ring induced by $\sigma \mapsto 
\sigma^{-1}$ for elements $\sigma$ in the group
as in \S \ref{MainConjecture}. 

\begin{Lemma} \label{ML2}
We assume $\mu=0$ for $K_{\infty}/k$. 
We have  
$$(\det f) \Lambda^-
=\theta_{K_{\infty},S'}^{T \hspace{0.5mm} \#}\Lambda^-$$ 
as ideals of $\Lambda^-$. 
\end{Lemma}

\begin{proof}
We write $\GA^{\circ}=\GA^{\circ}(K_{\infty}/k)^-_{\ZZ_{p}}$, 
and 
$Cl^{\vee}=((Cl_{K_{\infty},p}^{T})^{\vee})^-$. 
Since 
$$(W_{v}(K_{\infty}/k)^{\circ}_{\ZZ_{p}}/\Lambda e_{v}')^- 
\simeq \Lambda^-/(1-{\mathcal F}_{v}, \Delta I_{v}),$$
for $v \in S' \setminus (S_{\infty} \cup S_{p})$, the exact 
sequence (\ref{e5}) yields an exact sequence 
$$
0 \longrightarrow  \bigoplus_{v \in S' \setminus (S_{\infty} \cup S_{p})}
\Lambda^-/(1-{\mathcal F}_{v}, \Delta I_{v})
\longrightarrow \GA^{\circ}/\ii f 
\longrightarrow Cl^{\vee} \longrightarrow 0.
$$

Since $\Gal(K_{\infty}/k)$ is abelian and an extension of 
$\ZZ_{p}$ by a finite abelian group, we can write 
$\Gal(K_{\infty}/k) \simeq G' \times \ZZ_{p}$ for 
some finite subgroup $G'$. 
Let $K'$ be the field such that $\Gal(K_{\infty}/K')=\ZZ_{p}$, 
$\Gal(K'/k)=G'$, $K' \cap k_{\infty}=k$. 
By taking $K=K'$ from the first, 
we may assume $K \cap k_{\infty}=k$. 
Then $\Lambda$ is isomorphic to the power series ring 
$\ZZ_{p}[G][[t]]$. 

For an odd character $\chi$ of $G$, we consider the $\chi$-quotient 
only in the proof of this lemma. 
For a $\ZZ_{p}[G]$-module $M$ and 
$\chi: G \longrightarrow \overline{\QQ}_{p}^{\times}$ 
which is a character of $G$, whose image is in an algebraic 
closure of $\QQ_{p}$, we define the $\chi$-quotient 
$[M]_{\chi}$ by $[M]_{\chi}=M \otimes_{\ZZ_{p}[G]} \OO_{\chi}$
where $\OO_{\chi}=\ZZ_{p}[\ii \chi]$ on which 
$G$ acts via $\chi$. 
For an element $x$ of $M$, the image of $x$ in $[M]_{\chi}$ 
is denoted by $x_{\chi}$. 

Taking the $\chi$-quotients of the above exact sequence, we get 
an exact sequence
$$
\bigoplus_{v \in S' \setminus (S_{\infty} \cup S_{p})}
[\Lambda^-/(1-{\mathcal F}_{v}, \Delta I_{v})]_{\chi}
\longrightarrow [\GA^{\circ}/\ii f]_{\chi} 
\longrightarrow [Cl^{\vee}]_{\chi} \longrightarrow 0.
$$
The kernel of the first map is finite
since 
$$
\bigoplus_{v \in S' \setminus (S_{\infty} \cup S_{p})}
[(\Lambda^-/(1-{\mathcal F}_{v}, \Delta I_{v})) 
\otimes_{\ZZ_{p}} \QQ_{p}]_{\chi}
\longrightarrow [(\GA^{\circ}/\ii f) \otimes_{\ZZ_{p}} \QQ_{p}]_{\chi}
$$
is injective.  
We consider the characteristic ideals over $[\Lambda]_{\chi}=
\OO_{\chi}[[t]]$. 
We know $\cc([\GA^{\circ}/\ii f]_{\chi})
=((\det f)_{\chi})$.
%where $(\det f)_{\chi}$ is the image of $f$ in $[\Lambda]_{\chi}$.
If $\chi$ is trivial on $I_{v}$, we have
$\cc([\Lambda^-/(1-{\mathcal F}_{v}, \Delta I_{v})]_{\chi})
=((1-{\mathcal F}_{v})_{\chi})$. 
Otherwise, $[\Lambda^-/(1-{\mathcal F}_{v}, \Delta I_{v})]_{\chi}$ 
is finite.

Let $K_{\chi}$ be the intermediate field of $K/k$ 
corresponding to $\Ker \chi$,
and $K_{\chi \hspace{0.6mm} \infty}$ 
its cyclotomic $\ZZ_{p}$-extension. 
Then the characteristic ideal of $[Cl^{\vee}]_{\chi}$ 
is generated by 
$(\theta_{K_{\chi \hspace{0.5mm} \infty}}^{T \hspace{0.5mm} \#})_{\chi}$
by the main conjecture proved by Wiles \cite{Wiles}.
Therefore, the above exact sequence implies that  
$$
\cc([\GA^{\circ}/\ii f]_{\chi})
=((\det f)_{\chi})
=((\prod_{\chi_{\mid I_{v}=1}} (1-{\mathcal F}_{v})_{\chi}) 
(\theta_{K_{\chi \hspace{0.5mm} \infty}}^{T \hspace{0.5mm} \#})_{\chi})
$$
where $v$ ranges over all primes in $S'$ which are unramified in 
$K_{\chi \hspace{0.6mm} \infty}$. 
Let 
$${\rm res}_{K_{\chi \hspace{0.6mm} \infty}}:
\Lambda=\ZZ_{p}[[\Gal(K_{\infty}/k)]]
\longrightarrow 
\ZZ_{p}[[\Gal(K_{\chi \hspace{0.6mm} \infty}/k)]]
$$
be the restriction map. 
Since we know 
$$ {\rm res}_{K_{\chi \hspace{0.6mm} \infty}}
(\theta_{K_{\infty},S'}^{T \hspace{0.5mm} \#})
=
\prod_{\chi_{\mid I_{v}=1}} (1-{\mathcal F}_{v})
\theta_{K_{\chi \hspace{0.5mm} \infty}}^{T \hspace{0.5mm} \#},
$$
we obtain
$$
((\det f)_{\chi})=
((\theta_{K_{\infty},S'}
^{T \hspace{0.5mm} \#})_{\chi})
$$
as ideals of $\OO_{\chi}[[t]]$. 
Since this equality holds for any odd character $\chi$ of $G$,
the conclusion of Lemma \ref{ML2} follows from the next lemma. 
\end{proof}

\begin{Lemma} 
Let $a$, $b$ be elements of $\Lambda$ such that 
the $\mu$-invariants of $a_{\chi}$ and $b_{\chi}$ 
in $\OO_{\chi}[[t]]$ are zero for any character $\chi$ 
of $G$.
If 
%$a_{\chi} \OO_{\chi}[[t]]=b_{\chi} \OO_{\chi}[[t]]$ 
$(a_{\chi})=(b_{\chi})$ 
holds as ideals of $\OO_{\chi}[[t]]$ for all $\chi$ of $G$, 
we get 
%$a \Lambda= b \Lambda$ 
$(a)=(b)$
as ideals of $\Lambda$.
\end{Lemma}

\begin{proof}
This lemma seems to be well-known, but we give here a proof.  
By Proposition 2.1 in \cite{BG2} we may assume that 
$a$, $b$ are distinguished polynomials in the sense of \cite{BG2}.
We write $a=bq+r$ for some $q \in \Lambda$ and some polynomial $r$
whose degree is smaller than the degree of $b$. 
Here, $\Lambda$ is semi local, and the degree means the vector 
of the degree of each component (see \cite{BG2} \S 2). 
The condition $(a_{\chi})=(b_{\chi})$ in $\OO_{\chi}[[t]]$ implies 
$r_{\chi}=0$ for any $\chi$, so we have $r=0$ and $(a) \subset (b)$. 
The converse is also true, and we get $(a)=(b)$. 
%The conditions on $a$ and $b$ imply that 
%$\Lambda/(a)$ and $\Lambda/(b)$ are isomorphic after taking 
%localizations at any height one
%prime. 
%This shows that they are isomorphic as $\Lambda$-modules. 
%This shows that $(a)=(b)$ as ideals of $\Lambda$.
%
%This lemma can be also proved by using Proposition 2.1 in 
%\cite{BG2}. 
\end{proof}

Now we can prove the main theorem in this section. 
Recall that $S^{{\rm non} \ p}_{\rm ram}=S_{\rm ram} \setminus S_{p}$. 

\begin{Theorem} \label{MT2}
Assuming $\mu=0$ for $K_{\infty}/k$, we have 
$$\Fitt_{\Lambda^-}(((Cl_{K_{\infty},p}^{T})^{\vee})^-)
=(\prod_{v \in S^{{\rm non} \hspace{0.5mm} p}_{\rm ram}}
(1, \frac{N_{I_{v}}}{1-{\mathcal F}_{v}})) 
\theta_{K_{\infty}}^{T \hspace{0.5mm} \#} \ .$$
\end{Theorem}

\begin{proof}
By Lemmas \ref{ML} and \ref{ML2}, 
we have
\begin{eqnarray*}
\Fitt_{\Lambda^-}(((Cl_{K_{\infty},p}^{T})^{\vee})^-)
& = & (\prod_{v \in S' \setminus (S_{\infty} \cup S_{p})}
(1, \frac{N_{I_{v}}}{1-{\mathcal F}_{v}})) \det f \\
& = & (\prod_{v \in S' \setminus (S_{\infty} \cup S_{p})}
(1, \frac{N_{I_{v}}}{1-{\mathcal F}_{v}})) 
\theta_{K_{\infty},S'}^{T \hspace{0.5mm} \#} \ .
\end{eqnarray*}
If $v$ is unramified, we know 
$(1, \frac{N_{I_{v}}}{1-{\mathcal F}_{v}}) (1-{\mathcal F}_{v})
=\Lambda$,
so using 
$$
\theta_{K_{\infty},S'}^{T \hspace{0.5mm} \#}
=\prod_{S' \setminus (S_{\infty} \cup S_{\rm ram})} (1-{\mathcal F}_{v})
\theta_{K_{\infty}}^{T \hspace{0.5mm} \#},
$$
we obtain  
$$
(\prod_{v \in S' \setminus (S_{\infty} \cup S_{p})}
(1, \frac{N_{I_{v}}}{1-{\mathcal F}_{v}})) 
\theta_{K_{\infty},S'}^{T \hspace{0.5mm} \#} 
=
(\prod_{v \in S^{{\rm non} \hspace{0.5mm} p}_{\rm ram}}
(1, \frac{N_{I_{v}}}{1-{\mathcal F}_{v}})) 
\theta_{K_{\infty}}^{T \hspace{0.5mm} \#}.$$
This completes the proof of Theorem \ref{MT2}.

\end{proof}

\begin{Remark} \label{FinalRemark}
\begin{rm}
(1) Greither and Popescu proved that 
$\theta_{K_{\infty}}^{T \hspace{0.5mm} \#}$ is in 
the Fitting ideal of $((Cl_{K_{\infty},p}^{T})^{\vee})^-$ 
in \cite{GreiPo}. The above theorem gives a refinement 
in the sense that it gives a 
full description of the Fitting ideal. \\
%Theorem \ref{MT2} should be compared with 
%the results in Greither \cite{Grei}. 
(2) The author obtained a similar result for 
the non-Teichm\"{u}ller character components 
of class groups with $T=\emptyset$, 
assuming Leopoldt's conjecture 
in \cite{Ku3} Theorem A.5. 
Theorem \ref{MT2} 
implies Theorem A.5 in \cite{Ku3} without assuming Leopoldt's conjecture
by choosing $T$ suitably.  
Thus Theorem \ref{MT2} is also a generalization of the main result in 
the Appendix in \cite{Ku3}.\\
%(3) In the number field case, we need Conjecture \ref{C2} because of
%the lack of the corresponding result to Lemma \ref{ML2} in that case.
(3) When we study the non-Teichm\"{u}ller character components 
of the class groups (and the $T$-modified class groups), 
we saw that the duals of the class groups are
suitable objects for studying their Galois module structure 
in our previous papers (see \cite{Ku3}, \cite{Grei4}, \cite{Grei},
\cite{GreiKuri}, for example). 
One can see by Proposition \ref{P2} in this paper why the dual of 
the class group is relatively easier to handle than the class group 
itself. 
Concerning the study on the 
dual of the Teichm\"{u}ller character components, 
see \cite{GKT} and \cite{GKK}.
\end{rm}
\end{Remark}

\end{document}